\documentclass{imsart}

\usepackage{natbib}     
\usepackage{amsthm}     
\usepackage{amsmath}    
\usepackage{amsfonts}   
\usepackage{amssymb}    
\usepackage{mathrsfs}   
\usepackage{parskip}    
\usepackage{calc}       
\usepackage{enumitem}   
\usepackage{graphicx}   
\usepackage{lpic}       
\usepackage[english]{babel}
\RequirePackage{cleveref}    
\crefformat{assumption}{assumption #2#1#3}
\crefrangeformat{assumption}{assumptions~#3#1#4 to~#5#2#6}
\crefmultiformat{assumption}{assumptions~#2#1#3}{ and~#2#1#3}{, #2#1#3}{ and~#2#1#3}
\crefrangemultiformat{assumption}{assumptions~#3#1#4 to #5#2#6}{ and #3#1#4 to #5#2#6}{, #3#1#4 to #5#2#6}{ and #3#1#4 to #5#2#6}

\startlocaldefs

\newcommand{\cf}{{\it c.f.}}
\newcommand{\eg}{{\it e.g.}}

\newcommand{\lhs}{{\it l.h.s.}}
\newcommand{\rhs}{{\it r.h.s.}}

\newcommand{\etc}{{\it etcetera}}

\DeclareMathOperator{\weg}{\backslash}

\newcommand{\RR}{{\mathbb R}}

\newcommand{\NN}{{\mathbb N}}

\newcommand{\scrB}{{\mathscr B}}

\newcommand{\scrF}{{\mathscr F}}
\newcommand{\scrG}{{\mathscr G}}

\newcommand{\scrP}{{\mathscr P}}

\newcommand{\scrX}{{\mathscr X}}

\newcommand{\samplen}{{X^{n}}}

\newcommand{\realizationn}{{x^{n}}}

\newcommand{\tht}{{\theta}}

\newcommand{\Tht}{{\Theta}}

\newcommand{\set}[1]{\left\{ #1 \right\}}

\newcommand{\ft}[2]{{\textstyle{\frac{#1}{#2}}}}

\newcommand{\conv}[1]%
  {{\mathrel{\,\xrightarrow{\widthof{\,#1\,}}\,}}}
\newcommand{\convas}[1]%
  {{\mathrel{\,\xrightarrow{\widthof{\,#1\text{-a.s.}\,}}\,}}}
\newcommand{\convprob}[1]%
  {{\mathrel{\,\xrightarrow{\widthof{\,#1\,}}\,}}}
\newcommand{\convweak}[1]%
  {{\mathrel{\,\xrightarrow{\widthof{\,#1\text{-w.}\,}}\,}}}

\newcommand{\twobytwo}[4]%
  {\left(\begin{array}{cc} #1 & #2 \\ #3 & #4 \end{array}\right)}
\newcommand{\twovec}[2]%
  {\left({\begin{array}{c} #1\\#2 \end{array}}\right)}

\newcommand{\ceiling}[1]{\left\lceil #1 \right\rceil}
\newcommand{\floor}[1]{\left\lfloor #1 \right\rfloor}

\newcommand{\rh}[1]{\left(#1\right)}

\newtheoremstyle{customtheorem}
  {0.5em}
  {0.2em}
  {\itshape}
  {}
  {\scshape}
  {}
  {1ex}
  {}

\theoremstyle{customtheorem}
\newtheorem{theorem}{Theorem}[section]
\newtheorem{lemma}[theorem]{Lemma}
\newtheorem{proposition}[theorem]{Proposition}
\newtheorem{corollary}[theorem]{Corollary}
\newtheorem{definition}[theorem]{Definition}
\newtheorem{assumption}[theorem]{Assumption}

\newtheoremstyle{customremark}
  {0.5em}
  {0.2em}
  {}
  {}
  {\scshape}
  {}
  {1ex}
  {}

\theoremstyle{customremark}
\renewenvironment{proof}{\par\noindent{\scshape Proof}\;}{\hfill\qedsymbol\par}
\newtheorem{remark}[theorem]{Remark}
\newtheorem{example}[theorem]{Example}

\endlocaldefs

\begin{document}

\begin{frontmatter}

\title{Uncertainty quantification and testing in a stochastic block model with two unequal communities}
\runtitle{Uncertainty quantification and testing for communities}

\begin{aug}
  \author{\fnms{J.} \snm{van Waaij}\thanksref{m2}
    \ead[label=e2]{jvw@math.ku.dk}}
  \and
  \author{\fnms{B. J. K.} \snm{Kleijn}\thanksref{m1}\ead[label=e1]{b.kleijn@uva.nl}}
  \runauthor{J. van Waaij and B. Kleijn}
  \address{\thanksmark{m2}
    Department of Mathematical Sciences,
    University of Copenhagen\\
    Universitetsparken 5,
    DK-2100 Copenhagen,
    Denmark\\
    \printead{e2}
  }
  \affiliation{Korteweg-de~Vries Institute for Mathematics,
    University of Amsterdam\\
    Department of Statistics, University of Padova}
  \address{\thanksmark{m1}
    Korteweg-de Vries Institute for Mathematics,
    University of Amsterdam\\
    P.O. Box 94248,
    1090 GE Amsterdam,
    The Netherlands\\
    \printead{e1}
  }
\end{aug}

\begin{abstract}
\indent 
We show posterior convergence for the community structure in the planted bi-section model, for several interesting priors. Examples include where the label on each vertex is iid Bernoulli distributed, with some parameter $r\in(0,1)$. The parameter $r$ may be fixed, or equipped with a beta distribution.  
We do not have constraints on the class sizes, which might be as small as zero, or include all vertices, and everything in between. This enables us to test between a uniform (Erd\H os-R\'enyi) random graph with no distinguishable community or the planted bi-section model. The exact bounds for  posterior convergence enable us to convert credible sets into confidence sets. Symmetric testing with posterior odds is shown to be consistent. 
%
%
\end{abstract}



\end{frontmatter}


\section{Communities in random graphs}
\label{sec:intro}

The stochastic block model is a generalization of the Erd\H os-R\'enyi
random graph model \citep{Erdos59} where two vertices are connect with probability $p_n$. 
 Stochastic block models
\citep{Holland83} are similar but
concern random graphs with vertices that belong to one of several
classes and edge probabilities that depend on those classes.
If we think of the graph $X^n$ as data and the class
assignments of the vertices as unobserved, an interesting
statistical challenge presents itself regarding estimation of
(and other forms of inference on) the vertices' class assignments,
a task referred to as \emph{community detection} \citep{Girvan02}.
The stochastic block model and its generalizations have
applications in physics, biology, sociology, image processing,
genetics, medicine, logistics, \etc\ and are widely
employed as canonical models for the study of clustering and community
detection \citep{Fortunato10,Abbe18}.

In an asymptotic sense one may wonder under which conditions
on edge probabilities, community detection can be done in a
`statistically consistent' way as the number of vertices $n$
grows; particularly, whether it is possible to estimate the
true class assignments correctly (\emph{exact recovery}),
or correctly for a fraction of the vertices that goes to one
(\emph{almost-exact recovery}), with high probability (see
definitions~\ref{def:exact} and~\ref{def:detect} for details).

Here and in \citep{Abbe16,Massoulie14,Mossel16}, the community
detection problem is studied in the context
of the so-called \emph{planted bi-section model}, which is a
stochastic block model with two classes, 
and edge probabilities $p_n$ (within-class) and $q_n$
(between-class). A famous sufficient
condition for exact recovery of the
class assignment in the planted bi-section model comes from
\citep{Dyer89}: if there exists a constant $A>1$ such that,
$p_n-q_n\geq A n^{-1}\log n$, then community detection by
minimization of the number of edges between estimated classes
achieves exact recovery.
In \citep{Decelle11a,Decelle11b}, it was conjectured that
almost-exact recovery is possible in block models,
if $n(p_n-q_n)^2 > 2 (p_n+q_n)$. \cite{Mossel16}  prove a
definitive assertion: \emph{almost-exact recovery} is
possible (by any estimator or algorithm), if and only if,
\begin{equation}
  \label{eq:MNSdetect}
  \frac{n(p_n-q_n)^2}{p_n+q_n}\to\infty.
\end{equation}
An analogous claim in the Chernoff-Hellinger phase was first
considered more rigorously in \citep{Massoulie14} and later
confirmed, both from a probabilistic/statistical perspective
in \citep{Mossel15,Mossel16}, and independently from an
information theoretic perspective in \citep{Abbe16}.
Defining $a_n$ and $b_n$ by $np_n=a_n\log n$ and $nq_n=b_n\log n$
and assuming that $C^{-1}\leq a_n, b_n\leq C$ for all but finitely
many $n\geq1$, the class assignment in the planted bi-section model
can be \emph{recovered exactly}, if and only if,
\begin{equation}
  \label{eq:mnscritical}
  (a_n+b_n-2\sqrt{a_nb_n}-2)\log n + \log\log n\to \infty,
\end{equation}
(see \citep{Mossel16}). 

Estimation methods used for the community detection problem include
spectral clustering (see \citep{Krzakala13} and many others),
maximization of the likelihood and other modularities
\citep{Girvan02,Bickel09,Choi12,Amini13}, semi-definite programming
\citep{Hajek16,Guedon16}, and penalized ML detection of communities
with minimax optimal misclassification ratio 
\citep{Zhang16,Gao17}.
More generally, we refer to \citep{Abbe18} and
the very informative introduction of \citep{Gao17}
for extensive bibliographies and a more comprehensive discussion.
Bayesian methods have been popular throughout, \eg\ the original
work \citep{Snijders97}, the
work of \citep{Decelle11a,Decelle11b} and
more recently, \citep{Suwan16}, based on an
empirical prior choice, and \citep{Mossel16b}.
The machine learners' interest in the stochastic block model has
generated a wealth of algorithms that estimate the class assignment.
We mention only maximization of the likelihood or other modularities
\citep{Girvan02,Bickel09} and refer to the discussions in
\citep{Zhang16,Gao17}. 

In this paper we derive  exact recovery by means of posterior convergence for an interesting family of priors. Under the prior each vertex has label $\theta_i$, which are i.i.d. Bernoulli distributed with some parameter $r\in(0,1)$. When $r=1/2$, this is the uniform prior on the parameter space. The parameter $r$ might be fixed or equipped with a beta distribution. 
Another prior that we consider is the uniform prior on the size of the classes, and conditionally on the class sizes, the uniform prior on all labelings with these class sizes.

 We do not put restrictions on the size of the classes, which might be everything between 0 and $n$. We show posterior convergence for all class labelings, regardless of the size of the classes. This is a new contribution to the literature. To our knowledge all papers on consistency for the planted bi-section model require that both classes are (approximately) of equal size. Our relaxation of this requirement, has as additional benefit, that it allows us to test between the Erd\H os-R\'enyi model (basically a planted bi-section model with a class of size 0 and a class of size $n$) and the planted bi-section model. Furthermore, our precise bounds for exact recovery enable us to calculate confidence levels of credible sets.   

In the sparse Chernoff-Hellinger phase, where the edge degrees grow logarithmically, we derive exact recovery under conditions on the sparsity that closely resemble \cref{eq:mnscritical}, however, \citep{Mossel16} assume that the two classes are both of size $n/2$, where we allow it to be everything between zero and $n$. 
We derive exact bounds on the expected posterior mass of the true parameter, which enables us to derive confidence levels for credible sets. 
Finally Bayesian testing with posterior odds is considered, where we show consistency for testing between different class sizes. 

In the even sparser Kersten-Stigum phase where the edge degree is constant we derive almost-exact recovery. Our condition on the sparsity is equivalent to the necessary and sufficient condition \cref{eq:MNSdetect}, which shows that our results are sharp. 
In this ultra sparse Kersten-Stigum phase  we need to enlarge the credible sets in order to convert them into confidence sets.

In \cref{sec:pbm} we describe the model and give a general theorem for posterior convergence in the planted bi-section model. In \cref{sec:priors} we describe the priors that we consider and in \cref{sec:PBMposterior} we derive posterior convergence under the different sparsity regimes. Confidence sets are considered in \cref{sec:pbmuncertainty} and hypothesis testing in \cref{sec:hypothesistesting}. The proofs are defered to the appendix. \Cref{app:defs} establishes notation and basic Bayesian definitions.

\section{The planted bi-section model}
\label{sec:pbm}

In a stochastic block model,
each vertex is assigned to one of $K\geq2$ classes through an
unobserved \emph{class assignment vector} $\theta$.
Each vertex belongs to a class and any edge occurs (independently
of others) with a probability depending on 
 whether vertices that it connects belong to the same class or not. In the \emph{planted bi-section
model}, there are only two classes ($K=2$). The smallest class has $m$ vertices, and the largest $n-m$. We denote by $\theta$ the class assignment vector with 
components $\theta_{1},\ldots,\theta_{n}\in\{0,1\}$), where 0 denotes the largest class and 1 the smallest. The total parameter space is 
$\Theta_n=\bigcup_{m=0}^{\floor{n/2}}\Theta_{n,m}$. 
By $\Theta_{n,m}, m<n/2$, 
we denote the subset of  $\theta\in \{0,1\}^{n}$ with $\sum_{i=1}^n\theta_i=m$. 
In order to guarantee 
identifiability, when $n$ is even, 
we denote by $\Theta_{n,n/2}$ 
all labels $\theta$ with $\sum_{i=1}^n\theta_i=n/2$ 
and $\theta_1=0$. (Because $\theta$ and $(1-\theta_1,\ldots,1-\theta_n)$ 
induce the same law, as we will see later.) Note that $\Theta_{n,m}$ has $\binom nm$ elements, and when $n$ 
is even, $\Theta_{n,n/2}$ has $\frac12\binom n{n/2}$ elements. The full parameter set $\Theta_n$ has $2^{n-1}$ elements.  It 
is noted explicitly that \(m=0\) is also allowed, which allows us to test between the Erd\H os-R\'enyi graph model and a bi-section model. In case \(m=0\), \(\Theta_{n,0}\) consist of only one element: the \(n\)-vector \((0,\ldots, 0)\).

The space in which the random graph $X^n$ takes its values 
is denoted by
$\scrX_n$ (\eg\ represented by its adjacency matrix
with entries $\{X_{ij}:1\leq i,j\le n\}$). 
The ($n$-dependent) probability of an edge occuring ($X_{ij}=1$)
between vertices $1\leq i,j\le n$ \emph{within the same class}
is denoted $p_n\in(0,1)$; the probability of an
edge \emph{between classes} is denoted $q_n\in(0,1)$,
\begin{equation}
  \label{eq:pbm}
  Q_{ij}(\theta):=P_{\theta,n}(X_{ij}=1)=\begin{cases}
    \,\,p_n,&\quad\text{if $\theta_{n,i}=\theta_{n,j}$,}\\
    \,\,q_n,&\quad\text{if $\theta_{n,i}\neq\theta_{n,j}$.}
  \end{cases}
\end{equation}
Note that if $p_n=q_n$, $X^n$ is the Erd\H os-R\'enyi graph
$G(n,p_n)$ and the class assignment $\theta_n\in\Theta_n$ is not
identifiable. 

The probability measure for the graph $X^n$ corresponding to
parameter \(\theta\) is denoted $P_{\theta}$. The likelihood is
given by,
\[
  p_{\theta}(X^n)=\prod_{i<j} Q_{i,j}(\theta)^{X_{ij}}
    (1-Q_{i,j}(\theta))^{1-X_{ij}}.
\]
For the \emph{sparse versions} of the planted bi-section model,
we also define edge probabilities that vanish with growing $n$:
take $(a_n)$ and $(b_n)$ such that $a_n\log n=np_n$ and
$b_n\log n=nq_n$ for the Chernoff-Hellinger phase; take $(c_n)$
and $(d_n)$ such that $c_n=np_n$ and $d_n=nq_n$ for the
Kesten-Stigum phase. The fact that we do not allow loops 
(edges
that connect vertices with themselves) leaves room for
$\frac12(n-1)n$
possible edges in the random graph $X^n$ observed at iteration
$n$.  

The statistical question of interest in this model is to
reconstruct the unobserved class assignment vectors $\theta_n$
\emph{consistently}, that is, (close to) correctly with
probability growing to one as $n\to\infty$. Consistency
can be stated in various ways, as defined below.
\begin{definition}
\label{def:exact}
For each $n\in\NN$, let $\theta_n\in\Theta_n$. An estimator sequence
$\hat{\theta}_n:\scrX_n\to\Theta_n$ is said to \emph{recover the
class assignment $\theta_n$ exactly} if,
\[
  P_{\theta_n}\bigl(\,\hat{\theta}_n(X^n)=\theta_n\,\bigr)\to1,
\]
as $n\to\infty$, that is, if $\hat{\theta}_n$ indicates the correct communities
with high probability. 
\end{definition}
We also relax this consistency requirement somewhat in the form
of the following definition, \cf\ \citep{Mossel16} and others: for
$n\geq1$ and two class assignments $\theta,\eta\in\Theta_n$,
let $k(\theta,\eta)=\#\set{i:\theta_i\neq \eta_i}=\sum_{i=1}^n|\theta_i-\eta_i|$.

Note that $\theta$ and $(1-\theta_1,\ldots,1-\theta_n)$ induce the same law, so that $\theta$ is close to $\eta$ when either $k(\theta,\eta)$ or $n-k(\theta,\eta)$ is small. This is reflected in the following definition.
\begin{definition}
\label{def:detect}
Let $\theta_n\in\Theta_n$ be given. An estimator sequence
$\hat{\theta}_n:\scrX_n\to\Theta_n$ is said to \emph{recover
$\tht_{0,n}$ almost-exactly}, 
if, for some sequence $\ell_n=o(n)$,
\[
  P_{\theta_{n}}\bigl(\,
    k(\hat{\theta}_{n},\theta_{n})\wedge(n-k(\hat{\theta}_{n},\theta_{n})) \le \ell_n
    \,\bigr)\to 1.
\]
We say that \emph{$\hat{\theta}_n$ recovers $\theta_{0,n}$
with error rate $\ell_n$}.
\end{definition}
Below, we specialize to the Bayesian approach:
we choose prior distributions $\pi_n$ for all $\Theta_n$, ($n\ge1$)
and calculate the posterior: denoting the likelihood
by $p_{\theta}(X^n)$, the posterior for a set $A\subset\Theta_n$ is given by
\[
  \Pi(A|X^n)={\displaystyle \sum_{\theta\in A}
    p_{\theta}(X^n)\, \pi_n(\theta)}
    \biggm/
  {\displaystyle \sum_{\theta\in\Theta_n}
    p_{\theta}(X^n)\, \pi_n(\theta)},
\]
where $\pi_n:\Theta_n\to[0,1]$ is the probability mass function for
the prior distribution $\Pi_n$ on $\Theta_n$. 


We make the following convenient assumption on the prior (which  always holds after removing parameters with zero prior mass from the parameter space). 
\begin{assumption}\label{eq:prior}
The prior mass function \(\pi_n\) of the prior \(\Pi_n\) on \(\Theta_n\) satisfies \(\pi_n(\theta)>0\), for all \(\theta\in\Theta_n\). 

\end{assumption}

The posterior distribution of a subset \(S\subseteq \Theta_n\) is given by
\[
\Pi_n(S\mid X^n)=\frac{\sum_{\theta\in S}\pi_n(\theta) p_{\theta}(X^n) }{\sum_{\theta\in \Theta_n} \pi_n(\theta)p_{\theta}(X^n) }.
\]

\begin{proposition}
	\label{prop:postconvset} For fixed \(n\), consider a prior probability mass function \(\pi_n\) on \(\Theta_n\) satisfying \cref{eq:prior}. 
	Suppose that for some \(\theta\in \Theta_n\), we observe a graph
	$X^n$ with $n$ vertices, distributed according to 
	$P_{\theta}.
	$
Let \(S\subseteq\Theta_n\weg\set{\theta}\) be non-empty.
For \(\eta\in S\), define \[
\begin{split}
D_{1}(\theta,\eta)&=\{(i,j)\in\{1,\ldots,n\}^2:\,i<j,\,
\theta_{i}=\theta_{j},\,
\eta_{i}\neq\eta_{j}\},\\
D_{2}(\theta,\eta)&=\{(i,j)\in\{1,\ldots,n\}^2:\,i<j,\,
\theta_{i}\neq\theta_{j},\,
\eta_{i}=\eta_{j}\}.
\end{split}
\]
 Then \(D_{1}(\theta,\eta)\) and \(D_{2}(\theta,\eta)\) are disjoint, and if 
 \[
 0<B\le \min_{\eta\in S}|D_{1}(\theta,\eta)\cup D_{2}(\theta,\eta)|,
 \]then  
\begin{equation}
	\label{eq:ordervstestpwr}
	P_{\theta}
	\Pi_n\bigl(S\bigm|X^n\bigr)
	\le \rho(p_n,q_n)^{B}\sum_{\eta\in S}\frac{\pi_n(\eta)^{1/2}}{\pi_n(\theta)^{1/2}},
	\end{equation}
\end{proposition}
where $\rho(p,q)$ is the Hellinger-affinity between two Bernoulli-distributions with
parameters $p$ and $q$, which is given by
\[
\rho(p,q)=p^{1/2}q^{1/2}+(1-p)^{1/2}(1-q)^{1/2}.
\] 
The proof is deferred to  \cref{app:proofofposteriorconvergencegeneralcase}.

\section{Prior}\label{sec:priors}

We consider  hierarchical priors, conditionally defined by first putting a prior $\pi_n(m)$ on the size of the smallest class,  and conditionally on $m$ a uniform prior on $\Theta_{n,m}$. So
\begin{equation}\label{eq:normalformpriors} 
\begin{split}
	m \sim  & \pi_n(m),\quad m=1,\ldots,\floor{n/2}\\
	\theta\mid m \sim  & \pi_n(\theta\mid m)=\frac1{|\Theta_{n,m}|}.
\end{split}
\end{equation}

This class of priors includes several interesting examples. 

\begin{example}\label{ex:bernoulli}
For $r\in(0,1)$, we consider the prior on $\theta=(\theta_1,\ldots,\theta_n)\in\Theta_n$, defined by
\begin{align}
	\theta_i \stackrel{iid}\sim\, & \text{Bernoulli}(r), \quad i=1,\ldots,n,
\end{align}
and next set $\theta=(1-\theta_1,\ldots,1-\theta_n)$, when $\sum_{i=1}^n
\theta_i>n/2$ or when $\sum_{i=1}^n \theta_i=n/2$ and $\theta_1=1$. 
In this case, $\pi_n(m)=|\Theta_{n,m}|\rh{r^m(1-r)^{n-m}+r^{n-m}(1-r)^{m}}$. So $m$ has the same distribution as $Y\wedge (n-Y)$, where $Y$ is   binomially distributed with parameters $n$ and $r$.  
When $r=1/2$, then $\pi_n(\theta)=2^{1-n}$, for each $\theta\in\Theta_n$, what corresponds to the uniform prior on $\Theta_n$. 
\end{example}

\begin{example}\label{ex:beta}
Let $\alpha,\beta>0$, and  consider \begin{align*}
	r \sim &\, \text{beta}(\alpha,\beta)\\
	\theta_i\mid r \stackrel{iid}\sim\, & \text{Bernoulli}(r), \quad i=1,\ldots,n,
\end{align*}
and next set $\theta=(1-\theta_1,\ldots,1-\theta_n)$, when  $\theta\notin\Theta_n$.
In this case  \begin{align*}
	\pi(m)= & |\Theta_{n,m}| \int_0^1 \rh{r^m(1-r)^{n-m}+r^{n-m}(1-r)^{m}}\frac{r^{\alpha-1}(1-r)^{\beta-1}}{B(\alpha,\beta)}dr\\
	= & |\Theta_{n,m}| \frac{B(m+\alpha,n-m+\beta)+B(n-m+\alpha,m+\beta)}{B(\alpha,\beta)}.
	\end{align*}
	Note that $\alpha=\beta=1$, corresponds to the uniform prior on $r$. 
\end{example}

\begin{example}\label{ex:uniformonm}
	If we a-prior believe that every class size is equally likely, we could choose $m\sim \text{unif}\set{0,\ldots,\floor{n/2}}$, so $\pi(m)=\frac1{1+\floor{n/2}}$. 
\end{example}

\section{Posterior concentration at the parameter}\label{sec:PBMposterior}

 In this section we are interested whether the posterior concentrates its mass on the true parameter $\theta$, or in the very sparse case, in a small neighbourhood around $\theta$.

\subsection{Exact recovery}

First we study exact recovery for the examples in \cref{sec:priors}.
\newcommand{\boundexactallpriors}{ne^{-(2c-g)n/4}e^{ne^{-cn/2}}}  
\begin{theorem}\label{thm:posteriorintheparameter}
	Suppose $X^n$ is generated according to $\theta\in\Theta_n$. For the prior defined in \cref{ex:bernoulli} with $r=1/2$ (i.e. the uniform prior on $\Theta_n$), we have when 	
	$-\log \rho(p_n,q_n)\ge \frac{\alpha_n\log n}{n}$ for some sequence $\alpha_n$, then \[P_{\theta_{0,n}}\Pi_n(\Theta_{n}\weg \set{\theta_{0,n}}\mid X^n) \le 2n^{1-\alpha_n/2}e^{n^{1-\alpha_n/2}} .\] 
	Hence posterior convergence is achieved when \[
	(\alpha_n-2)\log n\to\infty. 
	\]
	When $p_n=\frac{a_n\log n}n$ and $q_n=\frac{b_n\log n}n$, then \[
	\rh{\rh{\sqrt{a_n}-\sqrt{b_n}}^2-4-\frac1{2n}a_nb_n\log n}\log n\to \infty,
	\]
	is a sufficient condition for posterior convergence.

	In the dense phase, let $c\ge -\log \rho(p_n,q_n)$,  and  $g\ge 0$ a constant, so that $g\ge \log\rh{\frac r{1-r}\bigvee \frac{1-r}r}$ in \cref{ex:bernoulli}, $g=2+2\log 2$ in \cref{ex:beta}, and $g=1+\log 2$ in \cref{ex:uniformonm}. Then, for each of the three cases,
\begin{align*}
 		 P_{\theta_{0,n}}\Pi_n(\Theta_{n}\weg \set{\theta_{0,n}}\mid X^n) \le &\,2\sqrt2\boundexactallpriors.
 \end{align*}

%
%
\end{theorem}
The proof is deferred to \cref{app:proofpostconvergenceinoneparameter}.

In the setting of \cite{Mossel16}, where $C^{-1}<a_n,b_n<C$, for some constant $C>1$, the sufficient conditions for posterior convergence of the uniform prior translates to 
\begin{align*}
\rh{a_n+b_n-2\sqrt{a_nb_n}-4}\log n\to\infty,
\end{align*}
which implies \cref{eq:mnscritical}. Our conditions are slightly stronger than their condition, however we allow all  $ a_n,b_n$  for which $p_n$ and $q_n$ are  probabilities, and in our setting the size of the true class can be everything between zero and $n$. Mossel e.a. assume that $n$ is even and that both classes have class size exactly $n/2$, which is limited and unrealistic.

\subsection{Almost-exact recovery}

Note that $\theta$ and $(1-\theta_1,\ldots,1-\theta_n)$ induce the same likelihood. Hence elements $\eta$ with $k(\theta,\eta)=n-k$ are `equally close' to $\theta$ as elements $\eta$ with $k(\theta,\eta)=k$. Let \begin{equation}\label{eq:definitionBkn}
B_{k_n}=\set{\eta\in\Theta_n:k(\eta,\theta)\wedge(n-k(\theta,\eta))<   k_n}.
\end{equation}

\newcommand{\ratealmostexactintermsofrho}{e^{-\alpha_nn\rh{ \log \alpha_n +\beta_n/2 - 1 -g/\alpha_n}/4}}
\newcommand{\ratealmostexactintermsofcnanddn}{\exp\rh{-\frac14\alpha_ n n\rh{\log\alpha_n + \frac14\rh{\sqrt{c_n}-\sqrt{d_n}}^2 - \frac1{8n}c_nd_n-1- g/\alpha_n}}}
\begin{theorem}\label{lem:weakconvergenceuniformprior}
Let $k_n\ge \alpha_nn$,  $-\log\rho(p_n,q_n)\ge \frac{\beta_n}n$, and let $g\ge 0$ be a constant, so that $g\ge \log\rh{\frac r{1-r}\bigvee \frac{1-r}r}$ in \cref{ex:bernoulli}, $g=2+2\log 2$ in \cref{ex:beta}, and $g=1+\log 2$ in \cref{ex:uniformonm}. Then
 \begin{align*}
 		 P_{\theta_{0,n}}\Pi_n(\Theta_{n}\weg B_{k_n}\mid X^n)  \le &\,2\sqrt2 \ratealmostexactintermsofrho.
 \end{align*}
If, instead of the condition on $\rho(p_n,q_n)$,  $p_n=\frac{c_n}n$ and $q_n=\frac{d_n}n$, then 
 \begin{align*}
 		 &P_{\theta_{0,n}}\Pi_n(\Theta_{n}\weg B_{k_n}\mid X^n)  \\
 		 \le &\,2\sqrt2 \ratealmostexactintermsofcnanddn.
 \end{align*}
\end{theorem}
The proof is deferred to \cref{app:proofofKerstenStigumcase}.

Almost exact recovery is established when $P_{\theta_{0,n}}\Pi_n(\Theta_{n}\weg B_{k_n}\mid X^n) $ converges to zero, while $\alpha_n\downarrow0$, however slowly. Hence, in all our examples, almost exact recovery is established when $(\sqrt{c_n}-\sqrt{d_n})^2\to\infty $, however slowly. In \cref{app:proofofalmostexactcreterium} we show that this condition is equivalent to the necessary and sufficient condition of \cref{eq:MNSdetect}, which shows that our results are sharp. 

Note that we have the fastest convergence for the uniform prior on $\Theta_n$ (which corresponds to \cref{ex:bernoulli} with $r=1/2$), as we can choose $g=0$ in this case. When  $r\neq 1/2$, $g>0$.  

\section{Uncertainty quantification}
\label{sec:pbmuncertainty}

Conditionally on an observation \(X^n\), a credible set of credible level \(1-\gamma_n\) is a measurable subset \(D_n(X^n)\) of the parameter set with posterior mass at least \(1-\gamma_n\):
\[
\Pi_n(D_n(X^n)\mid X^n)\ge 1-\gamma_n.
\]
In our (discrete, finite) setting any set-valued map \(x^n\mapsto B_n(x^n)\subseteq\Theta_n\), the corresponding map \(x^n\mapsto \Pi_n(B_n(x^n)\mid x^n)\) is measurable and positive, and hence the integral  \(P_{\theta_0}\Pi_n(B_n(X^n)\mid X^n)\) is well-defined, see \cref{app:defs} for details. From this perspective, a credible set (of credible level $1-\gamma_n$) is a set-valued map \(x^n\mapsto D_n(x^n) \) satisfying \(\Pi_n(D_n(x^n)\mid x^n)\ge 1-\gamma_n,\) for every \(x^n\in\scrX_n\). 
In nonparametric setting, credible sets can have bad coverage: \cite{freedman1999} provides us with examples. However in this section we 
show that in the case of exact recovery credible sets cover the true parameter with high probability.
In case of almost exact recovery we make the credible sets larger in order to guarantee asymptotic coverage, using ideas of \cite{Kleijn20}.

\subsection{Confidence level of credible sets}

In the particular case of exact recovery, using the specific discrete nature of our model, we can lower bound the confidence level of the credible set.

\begin{lemma}\label{lem:crediblesettoconfidencesets}
	Suppose \(P_{\theta}\Pi_n(\set{\theta}\mid X^n)\ge 1-x_n\), where \(0<x_n<1\). Let \(\gamma_n\in(0,1)\) and \(D_n(X^n)\) a \( 1-\gamma_n\) credible set, i.e. \(\Pi_n(D_n(X^n)\mid X^n)\ge 1-\gamma_n\). Then \[P_{\theta}(\theta\in D_n(X^n))\ge  1-\frac1{1-\gamma_n}x_n.\] 
\end{lemma}
The proof is deferred to \cref{app:confidencesets}.
It turns out that the confidence level mostly depends on the rate of convergence and only weekly on the credible level.

As a corollary to \cref{thm:posteriorintheparameter,lem:crediblesettoconfidencesets} we have 
\begin{corollary}\label{cor:confidenceofcrediblesets}
	Suppose $X^n$ is generated according to $\theta\in\Theta_n$. 
	Let $D_n$ be a $1-\gamma_n$ credible set. 
	For the prior in \cref{ex:bernoulli} with $r=1/2$, we have when 	
	$-\log \rho(p_n,q_n)\ge \frac{\alpha_n\log n}{n}$, for some sequence $\alpha_n$, then \[P_{\theta}(\theta\in D_n)\ge 1- \frac2{1-\gamma_n} n^{1-\alpha_n/2}e^{n^{1-\alpha_n/2}} .\] 
	In the dense phase, when  $c\ge -\log \rho(p_n,q_n)$, let $g\ge 0$ be a constant, so that $g\ge \log\rh{\frac r{1-r}\bigvee \frac{1-r}r}$ in \cref{ex:bernoulli}, $g=2+2\log 2$ in \cref{ex:beta}, and $g=1+\log 2$ in \cref{ex:uniformonm}. Then 
\begin{align*}
&		 P_{\theta}(\theta\in D_n)\ge 1-\frac{2\sqrt2}{1-\gamma_n}\boundexactallpriors.
\end{align*} 
\end{corollary}

\subsection{Enlarged credible sets}

In the case of almost exact convergence, credible sets need to be enlarged, in order to make them asymptotic confidence sets. 

Let \(D_n(X^n)\) be a credible set. For a  nonnegative integer \(k_n\), we define the \(k_n\)-enlargement of \(D_n(X^n)\) to be the set  \[
C_n(X^n)=\set{\theta_n\in\Theta_n:\exists \eta_n\in D_n(X^n), m_n(\theta_n,\eta_n)< k_n}.
\]

Recall the definition of \(B_{k_n}(\theta)\) in \cref{eq:definitionBkn}, 
\[
B_{k_n}(\theta)=\set{\eta\in\Theta_n:k(\eta,\theta)\wedge(n-k(\theta,\eta))<   k_n}.
\]

We have the following result

\begin{lemma}\label{lem:coverageofenlargedcrediblesets}
	Suppose \(P_{\theta}\Pi_n(B_{n,k_n}(\theta)\mid X^n)\ge1-x_n\), \(0<x_n<1\). Let \(\gamma_n\in(0,1)\) and \(D_n(X^n)\) a \(1-\gamma_n\)-credible set, with \(k_n\)-enlargement \(C_n(X^n)\), then \[P_{\theta}(\theta\in C_n(X^n))\ge 1-\frac1{1-\gamma_n}x_n.\]
\end{lemma}
The proof is deferred to \cref{app:confidencesetsenlarged}.

As a corollary to \cref{lem:weakconvergenceuniformprior,lem:coverageofenlargedcrediblesets} we have 

\begin{corollary}
Suppose $X^n$ is generated according to $\theta\in\Theta_n$. 
	Let $D_n(X^n)$ be a $1-\gamma_n$ credible set and let $C_n(X^n)$ its $\ceiling{\gamma_nn}$ enlargement. 
	When $-\log \rho(p_n,q_n)\ge \frac{\beta_n}n$, $g\ge 0$ is a constant so that in \cref{ex:bernoulli},  $g\ge \log\rh{\frac r{1-r}\bigvee \frac{1-r}r}$, in \cref{ex:beta} $g=2+2\log2$ and in \cref{ex:uniformonm} $g=1+\log2$,  then 
	\[
	P_{\theta}(\theta\in C_n(X^n)) \ge 1 -\frac {2\sqrt2}{1-\gamma_n} e^{-\gamma_nn\rh{\log \gamma_n +\beta_n/2 - 1- g /\gamma_n }/4}. 
	\] 
	If, instead of the condition on $\rho(p_n,q_n)$, $p_n=\frac{c_n}n$ and $q_n=\frac{d_n}n$, then
	\begin{align*}
	& P_{\theta}(\theta\in C_n(X^n))  \\ 
	\ge &1-\frac{2\sqrt2}{1-\gamma_n} \exp\rh{-\frac14\gamma_ n n\rh{\log\gamma_n + \frac14\rh{\sqrt{c_n}-\sqrt{d_n}}^2 - \frac1{8n}c_nd_n-1- g/\gamma_n}}.
\end{align*} 
\end{corollary}

\section{Consistent hypothesis testing with posterior odds}\label{sec:hypothesistesting}

Besides parameter estimation, an interesting question is testing between two alternatives, whether the true parameter \(\theta\) is in the set \(A_n\) or in the set \(B_n\), where \(A_n,B_n\subseteq \Theta_n\) are disjoint non-random sets. In particular we consider symmetric testing between two alternatives 
\[
H_0:\theta\in A_n\quad \text{versus}\quad H_1:\theta\in B_n.
\]Taking, for example, \(A_n= \Theta_{n ,0}\) and \(B_n=\Theta_n\weg\Theta_{n ,0}\) allows us to test whether the data was generated from a  Erd\H os-R\'enyi model or the planted bi-section model.

We use posterior odds to test between the models, which is defined by \begin{equation}\label{eq:definitionofF}
F_n=\frac{\Pi_n(B_n\mid X^n)}{\Pi_n(A_n\mid X^n)}.
\end{equation}
Obviously, \(F_n<1\) counts as evidence in favour of \(H_0\) and \(F_n>1\) as evidence in favour of \(H_1\). In the following theorem we give sufficient conditions for this Bayesian test to be  valid in a frequentist sense.

\begin{theorem}\label{thm:posteriorodds}
	Let \(\theta\in \Theta_n\). 	When   
	\(P_{\theta}\Pi_n(A_n\mid X^n)\ge  1-a_n\), with \(0<  a_n<1\), 
	then
	\[
	P_{\theta}(F_n> t_n)\le 
	2a_n\rh{1+    \frac{1 }{t_n}}.
	\]
	If, in addition,  \(P_{\theta}\Pi_n(B_n\mid X^n)\le  b_n\), 
	then
	\[
	P_{\theta}(F_n> t_n)\le 2a_n+\frac{2b_n}{t_n}.\]
\end{theorem}
We defer the proof to \cref{app:postodds}.

Suppose one rejects the null-hypothesis when \(F_n>t_n\), for some \(t_n>0\). The first order error is when \(H_0\) is true, so \(\theta\) is in fact in \(A_n\), but \(H_0\) is rejected (so \(F_n>t_n\)). The probability of this error is bounded by the theorem above. The error of second kind is when in fact \(H_1\) is true, but \(H_0\) is not rejected. This probability is given by \(P_{\theta}(F_n<t_n),\) \(\theta\in B_n\). As \(P_{\theta}(F_n<t_n)=P_{\theta}(F_n^{-1}>1/t_n),\) and reversing the roles of \(A_n \)
and \(B_n\) in \cref{thm:posteriorodds}, the probability of this event is also covered by the theorem, using posterior convergence results for \(\theta\in B_n\). The power of the test is defined as the probability of rejecting the null hypothesis when \(H_1\) is true.
As 
\(P_{\theta_{0, n}}(F_n>t_n)=P_{\theta_{0, n}}(F_n^{-1}<1/t_n)=1-P_{\theta_{0, n}}(F_n^{-1}\ge 1/t_n)\), this probability can be bounded from below with the theorem above. 

We have the following interesting corollary. 

\begin{corollary}
Let $m_0,m_1\in\set{0,\ldots,\floor{n/2}}$, $m_0\neq m_1$. 
	Suppose $X^n$ is generated according to $\theta\in\Theta_{n,m_0}$.
	Consider the test \[	H_0:\theta\in\Theta_{n,m_0}\quad \text{versus}\quad H_1:\theta\in\Theta_{n,m_1}.
	\]
	 For the prior in \cref{ex:bernoulli} with $r=1/2$, we have when 	
	$-\log \rho(p_n,q_n)\ge \frac{\alpha_n\log n}{n}$ for some sequence $\alpha_n$, then \[P_{\theta_{0,n}}(F_n>t_n)\le 4n^{1-\alpha_n/2}e^{n^{1-\alpha_n/2}}\rh{1+\frac1{t_n}}.\]
	
In the dense phase, when $c\ge -\log\rho(p_n,q_n)$, let $g\ge 0$ be a constant, so that $g\ge \log\rh{\frac r{1-r}\bigvee \frac{1-r}r}$ in \cref{ex:bernoulli}, $g=2+2\log 2$ in \cref{ex:beta}, and $g=1+\log 2$ in \cref{ex:uniformonm}.  Then
	\[P_{\theta}\Pi_n(F_n>t_n) \le 4\sqrt2\boundexactallpriors \rh{1+\frac1{t_n}}..
	\]
	The same bounds hold when we replace $H_1$ by $H_1:\theta\notin \Theta_{n,m_0}$. 
\end{corollary}
\begin{proof}
	Note that when $\theta\in\Theta_{n,m_0}$, then 
	\begin{align*}
		P_\theta \Pi_n(\Theta_{n,m_0}\mid X^n)\ge P_\theta \Pi_n(\set\theta\mid X^n). 
	\end{align*}
	The results now follow from \cref{thm:posteriorintheparameter,thm:posteriorodds}.
\end{proof}

%


\appendix


\section{Definitions and conventions}
\label{app:defs}

Because we take the perspective of a frequentist using Bayesian
methods, we are obliged to demonstrate that Bayesian definitions
continue to make sense under the assumption that the data $X^n$ is
distributed according to a true, underlying $P_0$.

\begin{remark}
\label{rem:conv}
We assume given for every $n\geq1$, a random graph
$\samplen$ taking values in the (finite) space
$\scrX_n$ of all undirected graphs with $n$ vertices.
We denote the powerset of $\scrX_n$ by $\scrB_n$ and regard it as
the domain for probability distributions
$P_n:\scrB_n\to[0,1]$ a model $\scrP_n$ parametrized by
$\Theta_n\rightarrow\scrP_n:\theta\mapsto P_{\theta,n}$
with finite parameter spaces $\Theta_n$ (with powerset $\scrG_n$)
and uniform priors
$\Pi_n$ on $\Tht_n$. As frequentists, we assume that there exists
a `true, underlying distribution for the data'; in this case, that
means that for every $n\geq1$, there exists a $\tht_{0,n}\in\Tht_n$
and corresponding $P_{\tht_0,n}$ from which the
$n$-th graph $\samplen$ is drawn.
\end{remark}
\begin{definition}
Given $n\geq1$ and a prior probability measure $\Pi_n$ on
$\Tht_n$, define the \emph{$n$-th prior predictive distribution}
as:
\begin{equation}
  \label{eq:priorpred}
  P_n^{\Pi}(A) = \int_\Theta P_{\theta,n}(A)\,d\Pi_n(\theta),
\end{equation}
for all $A\in\scrB_n$. For any $B_n\in\scrG_n$ with $\Pi_n(B_n)>0$,
define also the \emph{$n$-th local prior predictive distribution},
\begin{equation}
  \label{eq:localpriorpred}
  P_n^{\Pi|B}(A) = \frac{1}{\Pi_n(B_n)}\int_{B_n}
    P_{\theta,n}(A)\,d\Pi_n(\theta),
\end{equation}
as the predictive distribution on $\scrX_n$ that results from the prior
$\Pi_n$ when conditioned on $B_n$.
\end{definition}
The prior predictive distribution $P_n^{\Pi}$ is the marginal
distribution for $\samplen$ in the Bayesian perspective that
considers parameter and sample jointly
$(\theta,\samplen)\in\Theta\times\scrX_n$
as the random quantity of interest. 
\begin{definition}
\label{def:posterior}
Given $n\geq1$,
a \emph{(version of) the posterior} is any set-function
$\scrG_n\times\scrX_n\rightarrow[0,1]:
(A,x^n)\mapsto\Pi(\,A\,|\samplen=x^n)$
such that,
\begin{enumerate}
\item for $B\in\scrG_n$, the map
  $\realizationn\mapsto\Pi(B|\samplen=\realizationn)$ is
  $\scrB_n$-measurable,
\item for all $A\in\scrB_n$ and $V\in\scrG_n$,
  \begin{equation}
    \label{eq:disintegration}
    \int_A\Pi(V|\samplen)\,dP_n^{\Pi} = 
    \int_V P_{\theta,n}(A)\,d\Pi_n(\theta).
  \end{equation}
\end{enumerate}
\end{definition}
Bayes's Rule is expressed through equality~(\ref{eq:disintegration})
and is sometimes referred to as a `disintegration' (of the joint
distribution of $(\theta,X^n)$). 
Because the models $\scrP_n$ are dominated (denote the density
of $P_{\theta,n}$ by $p_{\theta,n}$), the fraction of integrated likelihoods,
\begin{equation}
  \label{eq:posteriorfraction}
  \Pi(V|\samplen)= 
  {\displaystyle{\int_V p_{\theta,n}(\samplen)\,d\Pi_n(\theta)}} \biggm/
  {\displaystyle{\int_{\Theta_n} p_{\theta,n}(\samplen)\,d\Pi_n(\theta)}},
\end{equation}
for $V\in\scrG_n$, $n\geq1$ defines a regular version of the
posterior distribution.

For completeness sake, we include \cite[lemma 2.2]{Kleijn20}, which plays an essential role in our theorems on  posterior consistency.

\begin{lemma}\label{lem:posteriorconvergencekleijnlemma}
	For any \(B_n, V_n\in \scrG_n\) with \(\Pi_n(B_n)>0\) and any measurable map \(\phi_n:\scrX_n \to[0,1] \), \begin{align*}
	\int P_{\theta_n}\Pi_n(V_n\mid X^n)d\Pi_n(\theta_n\mid B_n)\le &  \int P_{\theta_n}\phi_n(X^n)d\Pi_n(\theta_n\mid B_n) \\
	& + \frac1{\Pi_n(B_n)} \int _{V_n}P_{\theta_n}(1-\phi_n(X^n))d\Pi_n(\theta_n). 
	\end{align*}
\end{lemma}

\subsubsection*{Notation and conventions}

Asymptotic statements that end in ``... with high probability''
indicate that said statements are true with probabilities that
grow to one. The abbreviations \lhs\ and \rhs\ refer to ``left-''
and ``right-hand sides'' respectively.
For given probability measures $P,Q$ on a measurable space
$(\Omega,\scrF)$, we define
the Radon-Nikodym derivative $dP/dQ:\Omega\to[0,\infty)$,
$P$-almost-surely, referring {\it only} to the $Q$-dominated
component of $P$, following \citep{LeCam86}. We also \emph{define}
$(dP/dQ)^{-1}:\Omega\to(0,\infty]:\omega\mapsto 1/(dP/dQ(\omega))$,
$Q$-almost-surely.
Given random variables $Z_n\sim P_n$, weak convergence to a random
variable $Z$ is denoted by $Z_n\convweak{P_n}Z$, convergence
in probability by $Z_n\convprob{P_n}Z$ and almost-sure convergence
(with coupling $P^\infty$) by $Z_n\convas{P^{\infty}}Z$.
The integral of a
real-valued, integrable random variable $X$ with respect to a
probability measure $P$ is denoted $PX$, while integrals over
the model with respect to priors and posteriors are always written
out in Leibniz's or sum notation. The cardinality of a set $B$ is
denoted $|B|$.


\section{Existence of suitable tests}
\label{app:PBMtests}

Given $n\ge 1$, and two class assignment vectors
$\theta,\eta\in\Theta_n$,
we are interested in determining testing power, for which we need the likelihood ratio
$dP_{\eta}/dP_{\theta}$.

Fix $n\ge 1$, and let $X^n$ denote the random graph associated with
$\theta\in\Theta_n$, and let \(m_\theta\) be the number of 1-labels of \(\theta\), so \(\theta\in \Theta_{n,m_\theta}\). Let $\eta$ denote another element of $\Theta_n$ and  suppose \(\eta\in \Theta_{n ,m_\eta}\), for some \(m_\eta \in\set{0,\ldots,\floor{n/2}}\) (which might or might not be equal to \(m_\theta\)). 
Compare $p_{\theta}(X^n)$ with $p_{\eta}(X^n)$ in the likelihood
ratio. Recall that the likelihood of $\eta$ is given by,
\[
  p_{\eta}(X^n)=\prod_{i<j} Q_{i,j;n}(\eta)^{X_{ij}}
    (1-Q_{i,j;n}(\eta))^{1-X_{ij}},
\]
where 
\[
Q_{i,j;n}(\eta) = \begin{cases}
	\,\,p_n,&\quad\text{if $\eta_{i}=\eta_{j}$,}\\
	\,\,q_n,&\quad\text{if $\eta_{i}\neq\eta_{j}$.}
\end{cases}
\]
Consider 
\[
  \begin{split}
    D_{1}(\theta,\eta)&=\{(i,j)\in\{1,\ldots,n\}^2:\,i<j,\,
      \theta_{i}=\theta_{j},\,
      \eta_{i}\neq\eta_{j}\},\\
    D_{2}(\theta,\eta)&=\{(i,j)\in\{1,\ldots,n\}^2:\,i<j,\,
      \theta_{i}\neq\theta_{j},\,
      \eta_{i}=\eta_{j}\}.
  \end{split}
\]
 Also define,
\[
  (S_n,T_n):=\Bigl(\sum\{X_{ij}:(i,j)\in D_{1}(\theta,\eta)\},
    \sum \{X_{ij}:(i,j)\in D_{2}(\theta,\eta)\}\Bigr),
\]
and note that, under $P_{\theta}$ and $P_{\eta}$,
\begin{equation}
  \label{eq:SnTn}
  (S_n,T_n)\sim\begin{cases}
  \text{Bin}(|D_{1}(\theta,\eta)|,p_n)\times\text{Bin}(|D_{2}(\theta,\eta)|,q_n),
    \quad\text{if $X^n\sim P_{\theta}$},\\
  \text{Bin}(|D_{1}(\theta,\eta)|,q_n)\times\text{Bin}(|D_{2}(\theta,\eta)|,p_n),
    \quad\text{if $X^n\sim P_{\eta}$}.
  \end{cases}
\end{equation}
Since $S_n$ and $T_n$ are independent, the likelihood ratio is fixed
as a product two exponentiated binomial random variables: 
\begin{equation}
  \label{eq:pbmlikratio}
  \frac{p_{\eta}}{p_{\theta}}(X^n)
  = \biggl(\frac{1-p_n}{p_n}\,\frac{q_n}{1-q_n}\biggr)^{S_n-T_n}
    \biggl(\frac{1-q_n}{1-p_n}\biggr)^{|D_{1,n}|-|D_{2,n}|}
\end{equation} 
This gives rise to the following lemma:
\begin{lemma}
\label{lem:testingpower}
Let $n\geq1$, $\theta,\eta\in\Theta_n$ be given. Then there
exists a test function $\phi:\scrX_n\to[0,1]$ such that,
\begin{align*}
 & \pi_n(\theta)  P_{\theta}\phi(X^n) + \pi_n(\eta) P_{\eta}(1-\phi(X^n))\\
 \le & \pi_n(\theta)^{1/2}\pi_n(\eta)^{1/2}\rho(p_n,q_n)^{|D_{1,n}|+|D_{2,n}|}.
\end{align*}
\end{lemma}
where $\rho(p,q)$ is the Hellinger-affinity between two Bernoulli-distributions with
parameters $p$ and $q$, given by
\[
\rho(p,q)=p^{1/2}q^{1/2}+(1-p)^{1/2}(1-q)^{1/2}.
\] 
\begin{proof}
The likelihood ratio test $\phi(X^n)$ has testing power bounded by
the Hellinger transform,
\begin{align*}
&\pi_n(\theta)  P_{\theta}\phi(X^n) + \pi_n(\eta) P_{\eta}(1-\phi(X^n))\\
\le &  \pi_n(\theta)^{1/2}\pi_n(\eta)^{1/2}
    P_{\theta}\Bigl(\frac{p_{\eta}}{p_{\theta}}(X^n)\Bigr)^{1/2},
\end{align*}
(see, \eg\ \cite{LeCam86} and \cite[lemma 2.7]{Kleijn20}).
Then
\[
  \begin{split}
  P_{\theta}\biggl(\frac{p_{\eta}}{p_{\theta}}(X^n)\biggr)^{1/2}
    &= P_{\theta}
    \biggl(\frac{p_n}{1-p_n}\,\frac{1-q_n}{q_n}\biggr)^{\ft12(T_n-S_n)}
    \biggl(\frac{1-q_n}{1-p_n}\biggr)^{\ft12(|D_{1,n}|-|D_{2,n}|)}\\
    &= Pe^{\ft12\lambda_nS_n}\,Pe^{-\ft12\lambda_nT_n}
    \biggl(\frac{1-q_n}{1-p_n}\biggr)^{\ft12(|D_{1,n}|-|D_{2,n}|)},
  \end{split}
\]
where $\lambda_n:=\log(1-p_n)-\log(p_n)+\log(q_n)-\log(1-q_n)$ and
$(S_n,T_n)$ are distributed binomially, as in the first part of
(\ref{eq:SnTn}). Using the moment-generating function of the
binomial distribution, we conclude that,
\[
  \begin{split}
  &P_{\theta}\biggl(\frac{p_{\eta}}
    {p_{\theta}}(X^n)\biggr)^{1/2}
  =\Bigl(1-p_n
      +p_n\Bigl(\frac{1-p_n}{p_n}\,\frac{q_n}{1-q_n}\Bigr)^{1/2}
      \Bigr)^{|D_{1,n}|}\\
  &   \qquad\times\Bigl(1-q_n
      +q_n\Bigl(\frac{p_n}{1-p_n}\,\frac{1-q_n}{q_n}\Bigr)^{1/2}
      \Bigr)^{|D_{2,n}|}
      \biggl(\frac{1-q_n}{1-p_n}\biggr)^{\ft12(|D_{1,n}|-|D_{2,n}|)}\\
  &= \rho(p_n,q_n)^{|D_{1,n}|+|D_{2,n}|},
  \end{split}
\]
which proves the assertion. 
\end{proof}

\subsection{The sizes of \(D_{1}(\theta,\eta)\) and \(D_{2}(\theta,\eta)\)}

Note that $\set{1,\ldots,n}$ is the disjoint union of $V_{00}\cup V_{01}\cup V_{10}\cup V_{11}$ where \[
V_{ab} = \set{i: \theta_i=a,\eta_i=b}. 
\]
\begin{itemize}
	\item For $i \in V_{00}, j \in V_{01}$, $\theta_i=\theta_j=0$ and $\eta_i=0\neq 1=\eta_j$.
	\item For $i \in V_{01}, j \in V_{00}$, $\theta_i=\theta_j=0$ and $\eta_i=1\neq 0=\eta_j$.
	\item For $i \in V_{10}, j \in V_{11}$, $\theta_i=\theta_j=1$ and $\eta_i=0\neq 1=\eta_j$.
	\item For $i \in V_{11}, j \in V_{10}$, $\theta_i=\theta_j=1$ and $\eta_i=1\neq 0=\eta_j$. 
\end{itemize}
Note that in $D_{1}(\theta,\eta)$ we only count pairs $(i,j)$ with $i<j$, so
 \[
|D_{1}(\theta,\eta)| = \frac12\rh{2|V_{00}|\cdot |V_{01}|+ 2|V_{11}|\cdot |V_{10}|} = |V_{00}|\cdot |V_{01}|+ |V_{11}|\cdot |V_{10}|. 
\]

\begin{itemize}
	\item For $i \in V_{00}, j \in V_{10}$, $\theta_i=0\neq 1 =\theta_j$ and $\eta_i=\eta_j=0$. 
	\item For $i \in V_{10}, j \in V_{00}$, $\theta_i=1\neq 0 =\theta_j$ and $\eta_i=\eta_j=0$.
	\item For $i \in V_{01}, j \in V_{11}$, $\theta_i=0\neq 1 = \theta_j$ and $\eta_i=\eta_j=1$.
	\item For $i \in V_{11}, j \in V_{01}$, $\theta_i=1\neq 0 = \theta_j$ and $\eta_i=\eta_j=1$.
\end{itemize}
So, similar as with $D_{1}(\theta,\eta)$, \[
|D_{2}(\theta,\eta)| = |V_{00}|\cdot |V_{10}|+ |V_{01}|\cdot |V_{11}|. 
\]

So 
\[
|D_{1}(\theta,\eta)|+|D_{2}(\theta,\eta)| = |V_{00}|\big(|V_{01}|+|V_{10}|)+|V_{11}|\big(|V_{01}|+|V_{10}|\big)=\big(|V_{00}|+|V_{11}|\big)\big(|V_{01}|+|V_{10}|\big).
\]

Let $k=|V_{10}|+|V_{01}|$. 
Obviously, \[
|V_{00}|+|V_{01}|+|V_{10}|+|V_{11}|=n.
\]
So 
\[
|V_{00}|+|V_{11}|=n -k.
\]
So 
\begin{equation}\label{eq:sizeofD1andD2inVnmk}
d(\theta,\eta):=|D_{1}(\theta,\eta)|+ |D_{2}(\theta,\eta)|= k(n-k). 
\end{equation}

\section{Proofs}

\subsection{Proof of \cref{prop:postconvset}}\label{app:proofofposteriorconvergencegeneralcase}
Obviously, \(D_{1}(\theta,\eta)\) and \(D_{2}(\theta,\eta)\) are disjoint.
	According to  \cref{lem:posteriorconvergencekleijnlemma}  (with $B_n=\{\theta\}$),
	for any tests $\phi_{S}:\scrX_N\to[0,1]$,
	we have,
	\[
	P_{\theta}\Pi_n(S|X^n)
	\le P_{\theta}\phi_{S}(X^n)
	+\frac{1}{\pi_n(\theta)} \sum_{\eta\in S}\pi_n(\eta)
	P_{\eta}(1-\phi_{S}(X^n)).
	\]
	\Cref{lem:testingpower} proves that for any $\eta\in S$
	there is a test function $\phi_{\eta}$ that distinguishes
	$\theta$ from $\eta$ as follows,
	\begin{align*}
	  & P_{\theta}\phi_\eta(X^n) + \frac{\pi_n(\eta)}{\pi_n(\theta)} P_{\eta}(1-\phi_\eta(X^n))\\
	  \le &  \frac{\pi_n(\eta)^{1/2}}{\pi_n(\theta)^{1/2}}\rho(p_n,q_n)^{|D_{1}(\theta,\eta)|+|D_{2}(\theta,\eta)|}
	  \le \frac{\pi_n(\eta)^{1/2}}{\pi_n(\theta)^{1/2}}\rho(p_n,q_n)^{B_n},
	\end{align*}
	where the last inequality follows from the fact that \(\rho(p_n, q_n)\le 1\) and the assumption \(|D_{1}(\theta,\eta)\cup D_{2}(\theta,\eta)|\ge  B\), for all \(\eta\in S\). 
	Then using test functions
	$\phi_{S}(X^n)=\max\{\phi_{\eta}(X^n):\eta\in S\}$, 
	we have,
	\[
	P_{\theta}\phi_{S}(X^n)\le
	\sum_{\eta\in S}  P_{\theta}\phi_{\eta}(X^n),
	\]
	so that,
	\begin{align*}
	P_{\theta}\Pi_n(S|X^n) \le & \sum_{\eta\in S}  P_{\theta}\phi_{\eta}(X^n) + \frac{1}{\pi_n(\theta)} \sum_{\eta\in S}\pi_n(\eta)
	P_{\eta}(1-\phi_{S}(X^n))\\
	\le & \sum_{\eta\in S} \rh{P_{\theta}\phi_n(X^n) + \frac{\pi_n(\eta)}{\pi_n(\theta)} P_{\eta}(1-\phi_n(X^n))}\\
	\le &  \rho(p_n,q_n)^{B}\sum_{\eta\in S}\frac{\pi_n(\eta)^{1/2}}{\pi_n(\theta)^{1/2}}.
	\end{align*}

\subsection{Upper bound for $\rho(p_n,q_n)^{n/2}$}\label{app:proofsuffpostinchernoffhellingercase}

 Using \cref{lem:sqrtoneminusxislessthanorequaltooneminusxovertwo} we see 
\begin{align}
\rho(p_n,q_n) = & \sqrt{p_nq_n} + \sqrt{1-p_n}\sqrt{1-q_n}\nonumber \\
	\le & \sqrt{p_nq_n} + \rh{1-p_n/2}\rh{1-q_n/2} \nonumber\\ 
	= & 1  - \frac12\rh{\sqrt{p_n}-\sqrt{q_n}}^2 + p_nq_n/4\label{eq:boundforrho} \\ 
	= & 1 - \frac1n\rh{\frac12\rh{\sqrt{a_n}-\sqrt{b_n}}^2\log n - \frac1{4n}a_nb_n(\log n)^2} . \nonumber
\end{align}
So by \cref{lem:oneplusxdivrtothepowerrissmallerthanetothepowerx} \begin{align*}
	\rho(p_n,q_n)^{n/2} \le &  \rh{1 - \frac1{n/2}\rh{\frac14\rh{\sqrt{a_n}-\sqrt{b_n}}^2\log n - \frac1{8n}a_nb_n(\log n)^2} }^{n/2}\\
	\le &  \exp\rh{-\frac14\rh{\sqrt{a_n}-\sqrt{b_n}}^2\log n + \frac1{8n}a_nb_n(\log n)^2}.
\end{align*}

\subsection{Proof of \cref{thm:posteriorintheparameter}}\label{app:proofpostconvergenceinoneparameter}


 Define $V_{n,k}(\theta)=\set{\eta\in \Theta_{n}:k(\theta,\eta)=k}$. Note that for $\theta\in\Theta_n$,  $(1-\theta_1,\ldots,1-\theta_n)\notin \Theta_n$, hence $V_{n,k}$ is empty for $k\ge n$. Note that for $k=1,\ldots,n-1$ $V_{n,k}$ has at most $ \binom nk$ elements. It follows from \cref{eq:sizeofD1andD2inVnmk} that for all  $\eta\in V_{n,k}, |D_1(\theta,\eta)\cup D_2(\theta,\eta)|=k(n-k)$.

It follows from \cref{prop:postconvset}, that 
\begin{align}
	P_{\theta}
	\Pi_n\bigl(V_{n,k}\bigm|X^n\bigr)
	\le &  \rho(p_n,q_n)^{k(n-k)}\binom nk\sup_{\eta\in V_{n,k}}\sqrt{\frac{\pi_n(\eta)}{\pi_n(\theta)}}. \label{eq:upperboundforVnk}
\end{align}

For the uniform prior, so when $r=1/2$ in  \cref{ex:bernoulli}, $\pi_n(\eta)/\pi_n(\theta)=1$ for all $\theta,\eta\in \Theta_n$,  
so 
\begin{align*}
		 P_{\theta_{0,n}}\Pi_n(\Theta_{n}\weg \set{\theta_{0,n}}\mid X^n)= & \sum_{k=1}^{n-1} P_{\theta_{0,n}}\Pi_n(V_{n,k}(\theta)\mid X^n)\\
	\le & \sum_{k=1}^{n-1}\binom nk  \rho(p_n,q_n)^{k(n-k)}\\  \le & 2n\rho(p_n,q_n)^{n/2}e^{n\rho(p_n,q_n)^{n/2}},
\end{align*}
where we use \cref{eq:upperboundforVnk} for the first, and \cref{lem:boundforbinomialsum} for the second bound. Hence, when $-\log \rho(p_n,q_n)\ge \frac{\alpha_n\log n}{n}$ for some sequence $\alpha_n$, then \[P_{\theta_{0,n}}\Pi_n(\Theta_{n}\weg \set{\theta_{0,n}}\mid X^n) \le 2n^{1-\alpha_n/2}e^{n^{1-\alpha_n/2}} .\] 

Therefore posterior convergence is achieved once  \[(\alpha_n-2)\log n\to\infty.\]

%
The  sufficient condition for  $p_n=\frac{a_n\log n } n$ and $q_n=\frac{b_n\log n}n$ folows from \cref{app:proofsuffpostinchernoffhellingercase}.

When $\theta\in\Theta_{n,m_\theta}$ and $\eta\in\Theta_{n,m_\eta}$, then \begin{align*}
	 \frac{\pi_n(\eta)}{\pi_n(\theta)} 
	=  \frac{|\Theta_{n,m_\theta}|\pi_n(m_\eta)}{|\Theta_{n,m_\eta}| \pi_n(m_\theta)}
\end{align*}

So in \cref{ex:bernoulli}, for general $r\in(0,1),$
\begin{align*}
	 \frac{\pi_n(\eta)}{\pi_n(\theta)} = \frac{r^{m_\eta}(1-r)^{n-m_\eta}+r^{n-m_\eta}(1-r)^{m_\eta}}{r^{m_\theta}(1-r)^{n-m_\theta}+r^{n-m_\theta}(1-r)^{m_\theta}},
\end{align*}
which, according to \cref{lem:upperboundforr} is bounded by (and proportional to) 
\[
2\rh{\frac r{1-r}\bigvee \frac{1-r}r}^{m_\theta-m_\eta},
\]
which in turn is bounded by $2e^{fn/2}$. 

In \cref{ex:beta}, 
\begin{align*}
	 \frac{\pi_n(\eta)}{\pi_n(\theta)} = \frac{B(m_\eta+\alpha,n-m_\eta+\beta)+B(n-m_\eta+\alpha,m_\eta+\beta)}{B(m_\theta+\alpha,n-m_\theta+\beta)+B(n-m_\theta+\alpha,m_\theta+\beta)},
\end{align*}
which is bounded by $(2e)^n$, according to \cref{lem:boundonthefractionofthebetas}. 

Finally, in \cref{ex:uniformonm}, 
\begin{align*}
	 \frac{\pi_n(\eta)}{\pi_n(\theta)} = \frac{|\Theta_{n,m_\theta}|}{|\Theta_{n,m_\eta}|}\le \binom n{m_\theta}\le (2e)^{n/2}.
	 \end{align*} 
	 The result now follows from our previous calculations.  
%

\subsection{Proof of \cref{lem:weakconvergenceuniformprior}}\label{app:proofofKerstenStigumcase}

We use \cref{eq:upperboundforVnk}, the sets $V_{n,k}$ and the bounds of $\pi_n(\eta)/\pi_n(\theta)$ in \cref{app:proofpostconvergenceinoneparameter}.  Note that for \cref{ex:bernoulli,ex:beta,ex:uniformonm}, we have a bound $\pi_n(\eta)/\pi_n(\theta)\le 2C^n$, for some prior depend constant $C\ge 1$. 

%

By \cref{prop:postconvset}, 
 when $k_n\ge \alpha_nn$, we see that, 
\begin{align}
	 &P_{\theta_{0,n}}\Pi_n(\Theta_{n}\weg B_{k_n}\mid X^n) \nonumber \\
	 = &   \sum_{k=k_n}^{n-k_n}P_{\theta_{0,n}}\Pi_n(\Theta_{n}\weg B_{k_n}\mid X^n) \nonumber  \\
	\le  & \sqrt2C^{n/2}\cdot \sum_{k=k_n}^{n-k_n}\binom nk \rho(p_n,q_n)^{k(n-k)}\nonumber\\
	\le &  \sqrt2C^{n/2}\cdot \sum_{k=k_n}^{\floor{n/2}}\binom nk \rho(p_n,q_n)^{kn/2}\nonumber\\
	\le & \sqrt2C^{n/2} \cdot  \sum_{k=k_n}^\infty \rh{\frac{en}k}^k \rho(p_n,q_n)^{kn/2}\nonumber\\
	\le &  \sqrt2C^{n/2}\cdot\sum_{k=k_n}^\infty \rh{\frac{e}{\alpha_n}}^k \rho(p_n,q_n)^{kn/2}\nonumber\\
	\le & \sqrt2C^{n/2}\cdot  \frac{2\rh{\frac e{\alpha_n}\rho(p_n,q_n)^{n/2}}^{\alpha_nn}}{1-\frac e{\alpha_n}\rho(p_n,q_n)^{n/2}}.\label{eq:upperboundkerstenstigumphase}
	\end{align}
	
	Using \cref{lem:inequalityofanepowerdividedbyoneminusanepower}, and the fact that $-\log \rho(p_n,q_n)\ge \frac{\beta_n}n$,  \begin{align*}
		P_{\theta_{0,n}}\Pi_n(\Theta_{n}\weg B_{k_n}\mid X^n) 
	\le & 	 2\sqrt2C^{n/2} e^{-\alpha_nn\rh{ \log \alpha_n +\beta_n/2 - 1}/4}.
	\end{align*}
	The results now follow from the bounds on $\pi_n(\eta)/\pi_n(\theta)$ in \cref{app:proofpostconvergenceinoneparameter}. 

When $p_n=\frac{c_n}n$ and $q_n=\frac{d_n}n$, we get in a similar way as in \cref{app:proofsuffpostinchernoffhellingercase},
\begin{align*}
	\rho(p_n,q_n)^{n/2} 
	\le & \exp\rh{- \rh{\frac14\rh{\sqrt{c_n}-\sqrt{d_n}}^2 - \frac1{8n}c_nd_n}}.
\end{align*}

So \[
\frac e{\alpha_n}\rho(p_n,q_n)^{n/2} \le \exp\rh{1-\log\alpha_n - \rh{\frac14\rh{\sqrt{c_n}-\sqrt{d_n}}^2 - \frac1{8n}c_nd_n}}.
\]
Using \cref{eq:upperboundkerstenstigumphase} and \cref{lem:inequalityofanepowerdividedbyoneminusanepower}, we derive 
\begin{align*}
	& P_{\theta_{0,n}}\Pi_n(\Theta_{n}\weg B_{k_n}\mid X^n)  \\ 
	\le &2\sqrt2C^{n/2} \exp\rh{-\frac14\alpha_ n n\rh{\log\alpha_n + \frac14\rh{\sqrt{c_n}-\sqrt{d_n}}^2 - \frac1{8n}c_nd_n-1}}.
\end{align*}
The results now follow from the bounds on $\pi_n(\eta)/\pi_n(\theta)$ in \cref{app:proofpostconvergenceinoneparameter}.

\subsection{Equivalence of $(\sqrt{c_n}-\sqrt{d_n})^2\to\infty$ to \cref{eq:MNSdetect}}\label{app:proofofalmostexactcreterium}
Note that the necessary and sufficient condition \cref{eq:MNSdetect} translates to \[
\frac{(c_n-d_n)^2}{c_n+d_n}\to \infty. 
\]

Note that 
\begin{align*}
	& \frac{(c_n-d_n)^2}{c_n+d_n} = \frac{(\sqrt{c_n}-\sqrt{d_n})^2(\sqrt{c_n}+\sqrt{d_n})^2}{(\sqrt{c_n}+\sqrt{d_n})^2-2\sqrt{c_nd_n}}.
\end{align*}
Note that \begin{align*}
0 \le 	(\sqrt{c_n}-\sqrt{d_n})^2 = c_n+d_n -2\sqrt{c_nd_n}, 
\intertext{so}
\sqrt{c_nd_n}\le \frac{c_n+d_n}2,
\intertext{which is equivalent to }
2\sqrt{c_nd_n}\le \frac{c_n+d_n}2+\sqrt{c_nd_n}=\frac12(\sqrt{c_n}+\sqrt{d_n})^2.
\end{align*}

It follows that 
\begin{align*}
(\sqrt{c_n}-\sqrt{d_n})^2 \le \frac{(c_n-d_n)^2}{c_n+d_n} \le 2(\sqrt{c_n}-\sqrt{d_n})^2. 
\end{align*}

Hence 
\[
\frac{(c_n-d_n)^2}{c_n+d_n} \to\infty \quad\text{if and only if}\quad (\sqrt{c_n}-\sqrt{d_n})^2\to\infty. 
\]

\subsection{Confidence sets}

\begin{lemma}\label{lem1}
	Let \(n\ge1\). Let \(x^n\to B_n(x^n)\subset\Theta_n\) be a set valued map, such that  \(P_{\theta}\Pi_n(B_n(X^n)\mid X^n)\ge 1-a_n\), with \(0<a_n<1\).
	Then, for every \(0<r_n<1\), \[
	P_{\theta}\left(\Pi_n(B_n(X^n)\mid X^n)\ge 1-r_n\right)\ge  1-\frac1{r_n}a_n.\]
\end{lemma}
\begin{proof}	 
	Let \(E_n=\set{\omega:\Pi_n(B_n(X_n(\omega))\mid X^n(\omega))\ge 1-r_n}\) 
	be the event that the posterior mass of \(B_n(X^n)\) is at least \(1-r_n\). Let \(\delta>0\). Suppose that  \(P_{\theta}(E_{n})\le 1-\frac1{r_n}a_n-\delta\). Then \begin{align*}
	P_{\theta}\Pi_n(B_n(X^n)\mid X^{n})\le & P_{\theta}( E_{n})+(1-r_n)P_{\theta}(E_{n}^c)\\
	= &P_{\theta}( E_{n}) + (1-r_n) (1-P_{\theta}( E_{n}))\\
	= &r_nP_{\theta}( E_{n}) + 1-r_n \\
	\le & r_n\rh{1-\frac1{r_n}a_n-\delta} + 1-r_n\\
	= & 1-a_n-\delta r_n<1-a_n,
	\end{align*}
	which contradicts with our assumption that  \(P_{\theta}\Pi_n(B_n(X^n)\mid X^n)\ge 1-a_n\). Hence \(P_{\theta}(E_{n})> 1-\frac1{r_n}a_n-\delta\). As this holds for every \(\delta>0\), it follows that \(P_{\theta}(E_{n})\ge 1-\frac1{r_n}a_n\).
\end{proof}

\subsubsection{Proof of \cref{lem:crediblesettoconfidencesets}}\label{app:confidencesets}
	Let \(E_n=\set{\Pi(\set{\theta}\mid X^n)\ge r}\) be the event that \(\set{\theta}\) has posterior mass at least \(r\), \(r>\alpha_n\). 
	It follows from \cref{lem1} that \(P_{\theta}(E_n)\ge  1-\frac1{1-r}x_n\).
	As \(D_n(X^n)\) has at least \(1-\alpha_n\) posterior mass, \(D_n(X^n)\) and \(\set{\theta}\) cannot be disjoint on the event \(E_n\), as \(1-\alpha_n+r>1\). In other words, \(\theta\in D_n(X^n)\) on \(E_n\). So \(P_{\theta}(\theta\in D_n(X^n))\ge P_{\theta}(E_n)\ge  1-\frac1{1-r}x_n\). As this holds for any \(r>\alpha_n \) we have \(P_{\theta}(\theta\in D_n(X^n))\ge 1-\frac1{1-\alpha_n}x_n\).

\subsubsection{Proof of \cref{lem:coverageofenlargedcrediblesets}}\label{app:confidencesetsenlarged}
	Let \(E_n=\set{\Pi(B_{k_n}(\theta)\mid X^n)\ge r}\) be the event that \(B_{k_n}(\theta)\) has posterior mass at least \(r\), \(r>\alpha_n\). 
	It follows from \cref{lem1} that \(P_{\theta_{0, n}}(E_n)\ge  1-\frac1{1-r}x_n\). 
	As \(D_n(X^n)\) has at least \(1-\alpha_n\) posterior mass, \(D_n(X^n)\) and \(B_{k_n}(\theta)\) cannot be disjoint on the event \(E_n\), as \(1-\alpha_n+r>1\). Hence \(\theta\in C_n(X^n)\) on \(E_n\). So \(P_{\theta}(\theta\in C_n(X^n))\ge P_{\theta}(E_n)\ge  1-\frac1{1-r}x_n\). As this holds for any \(r>\alpha_n \) we have \(P_{\theta}(\theta\in C_n(X^n))\ge 1-\frac1{1-\alpha_n}x_n\).

\subsection{Proof of \cref{thm:posteriorodds}}\label{app:postodds}
	From the posterior convergence condition on \(A_n\) it follows that \(P_{\theta}\Pi_n(B_n\mid X^n)\le  a_n\). Hence the first result follows from the second, so we assume \(P_{\theta}\Pi_n(B_n\mid X^n)\le  b_n\) in what follows. 
	Let \(E_n=\set{\Pi_n(A_n\mid X_n)\ge 1/2}\) be the event that the posterior gives at least mass \(1/2\)  to \(A_n\). It follows from \cref{lem1} that \(P_{\theta}(E_n)\ge 1-2a_n.\)
	So	\begin{align*}
	P_{\theta}(F_n>r_n) 
	\le & P_{\theta} \rh{ \Pi_n(B_n\mid X_n)\ge t_n/2 } + 2a_n.
	\end{align*}  
	The probability on the right is by the Markov inequality bounded by \[\frac 2{t_n}P_{\theta}\Pi_n(B_n\mid X_n)\le \frac{2b_n}{t_n}.\] We thus arrive at the result
	\[
	P_{\theta}(F_n> t_n)\le 
	2a_n+\frac{2b_n}{t_n}.
	\]

\section{Auxiliary results}\label{sec:aux}

\begin{lemma}\label{lem:inequalityofanepowerdividedbyoneminusanepower}
	Let \(C\ge 2 \). For all  \(x\ge \sqrt{2/C}\),
	\begin{align*}
	\frac{e^{-Cx}}{1-e^{-x}}\le e^{-Cx/4}.
	\end{align*}
\end{lemma}
\begin{proof}
	Note that \begin{align*}
	1-e^{-x}&=\int_0^{x}e^{-y}dy\\
	&\ge xe^{-x}.
	\end{align*}
	Using this  and the fact that \(C-1\ge C/2\), we see that
	\begin{align*}
	\frac{e^{-Cx}}{1-e^{-x}}\le & \frac{e^{-Cx}}{xe^{-x}}= \frac{e^{-(C-1)x}}{x}\le \frac{e^{-Cx/2}}{x}=\frac{e^{-Cx/4}}{x}e^{-Cx/4}.
	\end{align*}
	As \(x\ge \sqrt{2/C}\)  is equivalent to \(Cx/4\ge x^{-1}/2\) and \(x>0\), we have 
	\begin{align*}
	\frac{e^{-Cx}}{1-e^{-x}}\le & x^{-1}e^{-x^{-1}/2}e^{-Cx/4}.
	\end{align*}
	One verifies that \(f(y)=ye^{-y/2}\) attains its maximum on \(\RR\) at \(y=2\) and \(f(2)=2/e<1\). It now follows that 
	\[\frac{e^{-Cx}}{1-e^{-x}}\le e^{-Cx/4}.\] 
\end{proof}

\begin{lemma}\label{lem:sqrtoneminusxislessthanorequaltooneminusxovertwo}
	For \(x\in[0,1]\), \(\sqrt{1-x}\le 1-x/2\). 
\end{lemma}
\begin{proof}
	Define \(f(x)=\sqrt{1-x}\) and \(g(x)=1-x/2\). Note that \(f(0)=g(0)\) and \(f'(x)=-\frac1{2\sqrt{1-x}}\le -\frac12 = g'(x)\), for all \(x\in[0,1]\). It follows that \(g(x)\ge f(x)\) for all \(x\in[0,1]\). 
\end{proof}

\begin{lemma}\label{lem:oneplusxdivrtothepowerrissmallerthanetothepowerx} 
	For all positive integers \(r\) and real numbers \(x>-r,\) \((1+x/r)^r \le e^x\).
\end{lemma}
\begin{proof}
	Let for \(x>-r\), \(f(x)= r\log(1+x/r)\) and \(g(x)=x\). Then \(f'(x)=\frac1{1+x/r}\) and \(g'(x)=1\). It follows that \(f'(x)\le g'(x)\), when \(x\ge 0\), \(f'(x)>g'(x) \) when \(-n<x<0\) and \(f(0)=g(0)\). It follows that \(f(x)\le g(x)\) for all \(x>-r\). As \(y\to e^y\) is increasing, for all real \(y\), it follows that for all  \(x>-n\), \((1+x/r)^r = e^{f(x)}\le e^{g(x)}=e^x\).
\end{proof}

\begin{lemma}\label{lem:boundforbinomialsum}
For $x\in[0,1]$, 
	\[
	\sum_{k=1}^{n-1} \binom nk x^{k(n-k)} \le 2\rh{(1+x^{n/2})^n -1} \le 2nx^{n/2}  e ^ {nx^{n/2}}.  
	\]
\end{lemma}
\begin{proof}
	Define $a_k=\binom nk x^{k(n-k)}$. Note that $a_k=a_{n-k}$. Using that $x\in[0,1]$ and the fact that $n-k\ge n/2$, for all $k\in\set{1,\ldots,\floor{n/2}}$, and the binomium of Newton, we see, 
	\begin{align*}
		\sum_{k=1}^{n-1} \binom nk x^{k(n-k)} 
		\le  &   2\sum_{k=1}^{\floor{n/2}} \binom nk x^{kn/2}\\
		\le  &   2\sum_{k=1}^{n} \binom nk x^{kn/2}\\
		= & 2\rh{(1+x^{n/2})^n -1}.
	\end{align*}
	
	Using \cref{lem:oneplusxdivrtothepowerrissmallerthanetothepowerx} we see 
	\begin{align*}
		(1+x^{n/2})^n = & \rh{1+\frac{nx^{n/2}}n}^n
		\le  e ^ {nx^{n/2}} . 	
	\end{align*}
	Using that $e^x-1\le xe^x$ for all $x\ge 0$, we see 
	\[
	\sum_{k=1}^n \binom nk x^{k(n-k)} \le 2nx^{n/2}  e ^ {nx^{n/2}} .
	\]
\end{proof}

\begin{lemma}\label{lem:upperboundforr}
For $r\in(0,1)$, and $m_1,m_2\in\set{0,\ldots,\floor{n/2}}$,  
	\begin{align*}
		\frac12\rh{\frac r{1-r}\bigvee \frac{1-r}r}^{m_2-m_1}\le \frac{r^{m_1}(1-r)^{n-m_1}+r^{n-m_1}(1-r)^{m_1}}{r^{m_2}(1-r)^{n-m_2}+r^{n-m_2}(1-r)^{m_2}}\\
		\le 2\rh{\frac r{1-r}\bigvee \frac{1-r}r}^{m_2-m_1}.
	\end{align*}
\end{lemma}
\begin{proof}
Note that
	\begin{align*}
		& \frac{{r^{m_1}(1-r)^{n-m_1}\vee r^{n-m_1}(1-r)^{m_1}}}{2\rh{{r^{m_2}(1-r)^{n-m_2}\vee r^{n-m_2}(1-r)^{m_2}}}}\\ 
		\le &\frac{r^{m_1}(1-r)^{n-m_1}+r^{n-m_1}(1-r)^{m_1}}{r^{m_2}(1-r)^{n-m_2}+r^{n-m_2}(1-r)^{m_2}} \\
		\le & \frac{2\rh{r^{m_1}(1-r)^{n-m_1}\vee r^{n-m_1}(1-r)^{m_1}}}{{r^{m_2}(1-r)^{n-m_2}\vee r^{n-m_2}(1-r)^{m_2}}}.
		\end{align*}
Let us calculate \begin{equation}\label{eq:proofofr}
	\frac{{r^{m_1}(1-r)^{n-m_1}\vee r^{n-m_1}(1-r)^{m_1}}}{{r^{m_2}(1-r)^{n-m_2}\vee r^{n-m_2}(1-r)^{m_2}}}.
\end{equation}
We consider the two cases. First suppose 
$r^{m_1}(1-r)^{n-m_1}\vee r^{n-m_1}(1-r)^{m_1}=r^{m_1}(1-r)^{n-m_1}$. In this case, \cref{eq:proofofr} is equal to 
\begin{align}
& {r^{m_1-m_2}(1-r)^{m_2-m_1}\wedge r^{m_1+m_2-n}(1-r)^{n-m_1-m_2}}\nonumber \\
= & {\rh{\frac r{1-r}}^{m_1-m_2}\bigwedge \rh{\frac r{1-r}}^{m_1+m_2-n}}. \label{eq:casesinr1}
\end{align}
When $r^{m_1}(1-r)^{n-m_1}\vee r^{n-m_1}(1-r)^{m_1}=r^{n-m_1}(1-r)^{m_1}$. Then \cref{eq:proofofr} is equal to 
\begin{align}
& {r^{n-m_1-m_2}(1-r)^{m_1+m_2-n}\wedge r^{m_2-m_1}(1-r)^{m_1-m_2}}\nonumber\\
= & {\rh{\frac r{1-r}}^{n-m_1-m_2}\bigwedge \rh{\frac r{1-r}}^{m_2-m_1}}. \label{eq:casesinr2}
\end{align}
We consider the case $r\le 1/2$ and $r>1/2$ seperately. First suppose that $r\le 1/2$. Note that for every choice of $m_1,m_2\in\set{0,\ldots,\floor{n/2}}$,  $m_1-m_2\ge m_1+m_2-n$, so \cref{eq:casesinr1} is equal to \[
\rh{\frac r{1-r}}^{m_1-m_2}
\]
and \cref{eq:casesinr2} is equal to 
\[
\rh{\frac r{1-r}}^{n-m_1-m_2}
\]
Using that for all $m_1,m_2\in \set{0,\ldots,\floor{n/2}}$, $m_1-m_2\le n-m_1-m_2$, we see that  \cref{eq:proofofr} is equal to \[
\rh{\frac r{1-r}}^{m_1-m_2}.\]

Now consider the case $r>1/2$. Then 
\cref{eq:casesinr1} is equal to
\[
\rh{\frac r{1-r}}^{m_1+m_2-n}\] and 
\cref{eq:casesinr2} is equal to
\[
\rh{\frac r{1-r}}^{m_2-m_1}.\]
So \cref{eq:proofofr} is equal to \[
\rh{\frac r{1-r}}^{m_2-m_1}.\]
So for all $r\in(0,1)$, 
\cref{eq:proofofr} is equal to \[
\rh{\frac r{1-r}\bigvee \frac{1-r}r}^{m_2-m_1}.\]
\end{proof}


\begin{lemma}\label{lem:boundonthefractionofthebetas}
	For $\alpha,\beta>0$, and $m_1,m_2\in\set{0,\ldots,\floor{n/2}}$, and $n\ge \alpha+\beta-2$, 
	\[
	\frac{B(m_1+\alpha, n-m_1+\beta)+B(n-m_1+\alpha, m_1+\beta)}{B(m_2+\alpha, n-m_2+\beta)+B(n-m_2+\alpha, m_2+\beta)}\le (2e)^{n}. 
	\]
\end{lemma}
\begin{proof}
	We have \begin{align*}
		& \frac{B(m_1+\alpha, n-m_1+\beta)+B(n-m_1+\alpha, m_1+\beta)}{B(m_2+\alpha, n-m_2+\beta)+B(n-m_2+\alpha, m_2+\beta)} \\
		= & \frac{\Gamma(m_1+\alpha)\Gamma( n-m_1+\beta)+\Gamma(n-m_1+\alpha)\Gamma( m_1+\beta)}{\Gamma(m_2+\alpha)\Gamma( n-m_2+\beta)+\Gamma(n-m_2+\alpha)\Gamma (m_2+\beta)}\\
		= &  \frac{ \frac{\Gamma(n+\alpha+\beta-1)}{\binom{n+\alpha+\beta-2}{m_1+\alpha - 1}} +\frac{\Gamma(n+\alpha+\beta-1)}{\binom{n+\alpha+\beta-2}{m_1+\beta - 1}}}{ \frac{\Gamma(n+\alpha+\beta-1)}{\binom{n+\alpha+\beta-2}{m_2+\alpha - 1}} +\frac{\Gamma(n+\alpha+\beta-1)}{\binom{n+\alpha+\beta-2}{m_2+\beta - 1}}}\\ 
		= &  \frac{ \frac{1}{\binom{n+\alpha+\beta-2}{m_1+\alpha - 1}} +\frac{1}{\binom{n+\alpha+\beta-2}{m_1+\beta - 1}}}{ \frac{1}{\binom{n+\alpha+\beta-2}{m_2+\alpha - 1}} +\frac{1}{\binom{n+\alpha+\beta-2}{m_2+\beta - 1}}}\\
		\le  &  \frac{2}{ 2(2e)^{-\ceiling{n+\alpha+\beta-2}/2}}\\  
		\le  & (2e)^{n}.
	\end{align*}
\end{proof}


\bibliographystyle{imsart-nameyear}
\bibliography{pbm3}

\begin{thebibliography}{26}

\bibitem[\protect\citeauthoryear{Abbe}{2018}]{Abbe18}
\begin{barticle}[author]
\bauthor{\bsnm{Abbe},~\bfnm{E.}\binits{E.}}
(\byear{2018}).
\btitle{Community Detection and Stochastic Block Models: Recent Developments}.
\bjournal{Journal of Machine Learning Research}
\bvolume{18}
\bpages{1-86}.
\end{barticle}
\endbibitem

\bibitem[\protect\citeauthoryear{Abbe, Bandeira and Hall}{2016}]{Abbe16}
\begin{barticle}[author]
\bauthor{\bsnm{Abbe},~\bfnm{E.}\binits{E.}},
  \bauthor{\bsnm{Bandeira},~\bfnm{{A. ~S. }}\binits{A.}} \AND
  \bauthor{\bsnm{Hall},~\bfnm{G.}\binits{G.}}
(\byear{2016}).
\btitle{Exact Recovery in the Stochastic Block Model}.
\bjournal{IEEE: Transactions on Information Theory}
\bvolume{62}.
\end{barticle}
\endbibitem

\bibitem[\protect\citeauthoryear{Amini et~al.}{2013}]{Amini13}
\begin{barticle}[author]
\bauthor{\bsnm{Amini},~\bfnm{{A. ~A. }}\binits{A.}},
  \bauthor{\bsnm{Chen},~\bfnm{A.}\binits{A.}},
  \bauthor{\bsnm{Bickel},~\bfnm{{P. ~J. }}\binits{P.}} \AND
  \bauthor{\bsnm{Levina},~\bfnm{E.}\binits{E.}}
(\byear{2013}).
\btitle{Pseudo-likelihood methods for community detection in large sparse
  networks}.
\bjournal{Ann. Statist.}
\bvolume{41}
\bpages{2097--2122}.
\bdoi{10.1214/13-AOS1138}
\end{barticle}
\endbibitem

\bibitem[\protect\citeauthoryear{Bickel and Chen}{2009}]{Bickel09}
\begin{barticle}[author]
\bauthor{\bsnm{Bickel},~\bfnm{{P. ~J. }}\binits{P.}} \AND
  \bauthor{\bsnm{Chen},~\bfnm{A.}\binits{A.}}
(\byear{2009}).
\btitle{A nonparametric view of network models and Newman-Girvan and other
  modularities}.
\bjournal{Proceedings of the National Academy of Sciences}
\bvolume{106}
\bpages{21068--21073}.
\bdoi{10.1073/pnas.0907096106}
\end{barticle}
\endbibitem

\bibitem[\protect\citeauthoryear{Choi, Wolfe and Airoldi}{2012}]{Choi12}
\begin{barticle}[author]
\bauthor{\bsnm{Choi},~\bfnm{{D. ~S. }}\binits{D.}},
  \bauthor{\bsnm{Wolfe},~\bfnm{{P. ~J. }}\binits{P.}} \AND
  \bauthor{\bsnm{Airoldi},~\bfnm{{E. ~M. }}\binits{E.}}
(\byear{2012}).
\btitle{Stochastic blockmodels with a growing number of classes}.
\bjournal{Biometrika}
\bvolume{99}
\bpages{273-284}.
\bdoi{10.1093/biomet/asr053}
\end{barticle}
\endbibitem

\bibitem[\protect\citeauthoryear{Decelle et~al.}{2011a}]{Decelle11a}
\begin{barticle}[author]
\bauthor{\bsnm{Decelle},~\bfnm{A.}\binits{A.}},
  \bauthor{\bsnm{Krzakala},~\bfnm{F.}\binits{F.}},
  \bauthor{\bsnm{Moore},~\bfnm{C.}\binits{C.}} \AND
  \bauthor{\bsnm{{Zdeborov\'a}},~\bfnm{L.}\binits{L.}}
(\byear{2011}a).
\btitle{Inference and Phase Transitions in the Detection of Modules in Sparse
  Networks}.
\bjournal{Phys. Rev. Lett.}
\bvolume{107}
\bpages{065701}.
\bdoi{10.1103/PhysRevLett.107.065701}
\end{barticle}
\endbibitem

\bibitem[\protect\citeauthoryear{Decelle et~al.}{2011b}]{Decelle11b}
\begin{barticle}[author]
\bauthor{\bsnm{Decelle},~\bfnm{A.}\binits{A.}},
  \bauthor{\bsnm{Krzakala},~\bfnm{F.}\binits{F.}},
  \bauthor{\bsnm{Moore},~\bfnm{C.}\binits{C.}} \AND
  \bauthor{\bsnm{{Zdeborov\'a}},~\bfnm{L.}\binits{L.}}
(\byear{2011}b).
\btitle{Asymptotic analysis of the stochastic block model for modular networks
  and its algorithmic applications}.
\bjournal{Phys. Rev. E}
\bvolume{84}
\bpages{066106}.
\bdoi{10.1103/PhysRevE.84.066106}
\end{barticle}
\endbibitem

\bibitem[\protect\citeauthoryear{Dyer and Frieze}{1989}]{Dyer89}
\begin{barticle}[author]
\bauthor{\bsnm{Dyer},~\bfnm{{M. ~E. }}\binits{M.}} \AND
  \bauthor{\bsnm{Frieze},~\bfnm{{A. ~M. }}\binits{A.}}
(\byear{1989}).
\btitle{The solution of some random NP-hard problems in polynomial expected
  time}.
\bjournal{Journal of Algorithms}
\bvolume{10}
\bpages{451 - 489}.
\bdoi{https://doi.org/10.1016/0196-6774(89)90001-1}
\end{barticle}
\endbibitem

\bibitem[\protect\citeauthoryear{{Erd\H os} and {R\'enyi}}{1959}]{Erdos59}
\begin{barticle}[author]
\bauthor{\bsnm{{Erd\H os}},~\bfnm{P.}\binits{P.}} \AND
  \bauthor{\bsnm{{R\'enyi}},~\bfnm{A.}\binits{A.}}
(\byear{1959}).
\btitle{On Random Graphs I}.
\bjournal{Publicationes Mathematicae}.
\end{barticle}
\endbibitem

\bibitem[\protect\citeauthoryear{Fortunato}{2010}]{Fortunato10}
\begin{barticle}[author]
\bauthor{\bsnm{Fortunato},~\bfnm{S.}\binits{S.}}
(\byear{2010}).
\btitle{Community detection in graphs}.
\bjournal{Physics Reports}
\bvolume{486}
\bpages{75 - 174}.
\bdoi{https://doi.org/10.1016/j.physrep.2009.11.002}
\end{barticle}
\endbibitem

\bibitem[\protect\citeauthoryear{Freedman}{1999}]{freedman1999}
\begin{barticle}[author]
\bauthor{\bsnm{Freedman},~\bfnm{D.}\binits{D.}}
(\byear{1999}).
\btitle{On the Bernstein-Von Mises Theorem with Infinite-Dimensional
  Parameters}.
\bjournal{The Annals of Statistics}
\bvolume{27}
\bpages{1119-1140}.
\end{barticle}
\endbibitem

\bibitem[\protect\citeauthoryear{Gao et~al.}{2017}]{Gao17}
\begin{barticle}[author]
\bauthor{\bsnm{Gao},~\bfnm{C.}\binits{C.}},
  \bauthor{\bsnm{Ma},~\bfnm{Z.}\binits{Z.}}, \bauthor{\bsnm{Zhang},~\bfnm{{A.
  ~Y. }}\binits{A.}} \AND \bauthor{\bsnm{Zhou},~\bfnm{{H. ~H. }}\binits{H.}}
(\byear{2017}).
\btitle{Achieving Optimal Misclassification Proportion in Stochastic Block
  Models}.
\bjournal{Journal of Machine Learning Research}
\bvolume{18}
\bpages{1-45}.
\end{barticle}
\endbibitem

\bibitem[\protect\citeauthoryear{Girvan and Newman}{2002}]{Girvan02}
\begin{barticle}[author]
\bauthor{\bsnm{Girvan},~\bfnm{M.}\binits{M.}} \AND
  \bauthor{\bsnm{Newman},~\bfnm{{M. ~E. ~J. }}\binits{M.}}
(\byear{2002}).
\btitle{Community structure in social and biological networks}.
\bjournal{Proceedings of the National Academy of Sciences of the United States
  of America}
\bvolume{99}
\bpages{7821-7826}.
\bdoi{10.1073/pnas.122653799}
\end{barticle}
\endbibitem

\bibitem[\protect\citeauthoryear{{Gu{\'e}don} and Vershynin}{2016}]{Guedon16}
\begin{barticle}[author]
\bauthor{\bsnm{{Gu{\'e}don}},~\bfnm{O.}\binits{O.}} \AND
  \bauthor{\bsnm{Vershynin},~\bfnm{R.}\binits{R.}}
(\byear{2016}).
\btitle{Community detection in sparse networks via Grothendieck's inequality}.
\bjournal{Probability Theory and Related Fields}
\bvolume{165}
\bpages{1025--1049}.
\bdoi{10.1007/s00440-015-0659-z}
\end{barticle}
\endbibitem

\bibitem[\protect\citeauthoryear{Hajek, Wu and Xu}{2016}]{Hajek16}
\begin{barticle}[author]
\bauthor{\bsnm{Hajek},~\bfnm{B.}\binits{B.}},
  \bauthor{\bsnm{Wu},~\bfnm{Y.}\binits{Y.}} \AND
  \bauthor{\bsnm{Xu},~\bfnm{J.}\binits{J.}}
(\byear{2016}).
\btitle{Achieving Exact Cluster Recovery Threshold via Semidefinite
  Programming}.
\bjournal{IEEE Trans. Inf. Theor.}
\bvolume{62}
\bpages{2788--2797}.
\bdoi{10.1109/TIT.2016.2546280}
\end{barticle}
\endbibitem

\bibitem[\protect\citeauthoryear{Holland, Laskey and
  Leinhardt}{1983}]{Holland83}
\begin{barticle}[author]
\bauthor{\bsnm{Holland},~\bfnm{{P. ~W. }}\binits{P.}},
  \bauthor{\bsnm{Laskey},~\bfnm{{K. ~B. }}\binits{K.}} \AND
  \bauthor{\bsnm{Leinhardt},~\bfnm{S.}\binits{S.}}
(\byear{1983}).
\btitle{Stochastic blockmodels: First steps}.
\bjournal{Social Networks}
\bvolume{5}
\bpages{109 - 137}.
\bdoi{https://doi.org/10.1016/0378-8733(83)90021-7}
\end{barticle}
\endbibitem

\bibitem[\protect\citeauthoryear{Kleijn}{2020}]{Kleijn20}
\begin{barticle}[author]
\bauthor{\bsnm{Kleijn},~\bfnm{{B. ~J. ~K. }}\binits{B.}}
(\byear{2020}).
\btitle{Frequentist validity of Bayesian limits}.
\bjournal{Ann. Statist. (accepted)}.
\end{barticle}
\endbibitem

\bibitem[\protect\citeauthoryear{Krzakala et~al.}{2013}]{Krzakala13}
\begin{barticle}[author]
\bauthor{\bsnm{Krzakala},~\bfnm{F.}\binits{F.}},
  \bauthor{\bsnm{Moore},~\bfnm{C.}\binits{C.}},
  \bauthor{\bsnm{Mossel},~\bfnm{E.}\binits{E.}},
  \bauthor{\bsnm{Neeman},~\bfnm{J.}\binits{J.}},
  \bauthor{\bsnm{Sly},~\bfnm{A.}\binits{A.}},
  \bauthor{\bsnm{{Zdeborov{\'a}}},~\bfnm{L.}\binits{L.}} \AND
  \bauthor{\bsnm{Zhang},~\bfnm{P.}\binits{P.}}
(\byear{2013}).
\btitle{Spectral redemption in clustering sparse networks}.
\bjournal{Proceedings of the National Academy of Sciences}
\bvolume{110}
\bpages{20935--20940}.
\bdoi{10.1073/pnas.1312486110}
\end{barticle}
\endbibitem

\bibitem[\protect\citeauthoryear{{Le~Cam}}{1986}]{LeCam86}
\begin{bbook}[author]
\bauthor{\bsnm{{Le~Cam}},~\bfnm{L.}\binits{L.}}
(\byear{1986}).
\btitle{Asymptotic methods in statistical decision theory}.
\bpublisher{Springer-Verlag New York}.
\bdoi{10.1007/978-1-4612-4946-7}
\end{bbook}
\endbibitem

\bibitem[\protect\citeauthoryear{Massouli{\'e}}{2014}]{Massoulie14}
\begin{binproceedings}[author]
\bauthor{\bsnm{Massouli{\'e}},~\bfnm{L.}\binits{L.}}
(\byear{2014}).
\btitle{Community detection thresholds and the weak Ramanujan property}.
In \bbooktitle{STOC 2014: 46th Annual Symposium on the Theory of Computing}
\bpages{1-10}.
\end{binproceedings}
\endbibitem

\bibitem[\protect\citeauthoryear{Mossel, Neeman and Sly}{2015}]{Mossel15}
\begin{barticle}[author]
\bauthor{\bsnm{Mossel},~\bfnm{E.}\binits{E.}},
  \bauthor{\bsnm{Neeman},~\bfnm{J.}\binits{J.}} \AND
  \bauthor{\bsnm{Sly},~\bfnm{A.}\binits{A.}}
(\byear{2015}).
\btitle{Reconstruction and estimation in the planted partition model}.
\bjournal{Probability Theory and Related Fields}
\bvolume{162}
\bpages{431--461}.
\bdoi{10.1007/s00440-014-0576-6}
\end{barticle}
\endbibitem

\bibitem[\protect\citeauthoryear{Mossel, Neeman and Sly}{2016a}]{Mossel16}
\begin{barticle}[author]
\bauthor{\bsnm{Mossel},~\bfnm{E.}\binits{E.}},
  \bauthor{\bsnm{Neeman},~\bfnm{J.}\binits{J.}} \AND
  \bauthor{\bsnm{Sly},~\bfnm{A.}\binits{A.}}
(\byear{2016}a).
\btitle{Consistency thresholds for the planted bisection model}.
\bjournal{Electron. J. Probab.}
\bvolume{21}
\bpages{24 pp.}
\bdoi{10.1214/16-EJP4185}
\end{barticle}
\endbibitem

\bibitem[\protect\citeauthoryear{Mossel, Neeman and Sly}{2016b}]{Mossel16b}
\begin{barticle}[author]
\bauthor{\bsnm{Mossel},~\bfnm{E.}\binits{E.}},
  \bauthor{\bsnm{Neeman},~\bfnm{J.}\binits{J.}} \AND
  \bauthor{\bsnm{Sly},~\bfnm{A.}\binits{A.}}
(\byear{2016}b).
\btitle{Belief propagation, robust reconstruction and optimal recovery of block
  models}.
\bjournal{Ann. Appl. Probab.}
\bvolume{26}
\bpages{2211--2256}.
\bdoi{10.1214/15-AAP1145}
\end{barticle}
\endbibitem

\bibitem[\protect\citeauthoryear{Nowicki and Snijders}{2001}]{Snijders97}
\begin{barticle}[author]
\bauthor{\bsnm{Nowicki},~\bfnm{K.}\binits{K.}} \AND
  \bauthor{\bsnm{Snijders},~\bfnm{{T. ~A. ~B. }}\binits{T.}}
(\byear{2001}).
\btitle{Estimation and Prediction for Stochastic Blockstructures}.
\bjournal{Journal of the American Statistical Association}
\bvolume{96}
\bpages{1077-1087}.
\bdoi{10.1198/016214501753208735}
\end{barticle}
\endbibitem

\bibitem[\protect\citeauthoryear{Suwan et~al.}{2016}]{Suwan16}
\begin{barticle}[author]
\bauthor{\bsnm{Suwan},~\bfnm{S.}\binits{S.}}, \bauthor{\bsnm{Lee},~\bfnm{{D.
  ~S. }}\binits{D.}}, \bauthor{\bsnm{Tang},~\bfnm{R.}\binits{R.}},
  \bauthor{\bsnm{Sussman},~\bfnm{{D. ~L. }}\binits{D.}},
  \bauthor{\bsnm{Tang},~\bfnm{M.}\binits{M.}} \AND
  \bauthor{\bsnm{Priebe},~\bfnm{{C. ~E. }}\binits{C.}}
(\byear{2016}).
\btitle{Empirical Bayes estimation for the stochastic blockmodel}.
\bjournal{Electron. J. Statist.}
\bvolume{10}
\bpages{761--782}.
\bdoi{10.1214/16-EJS1115}
\end{barticle}
\endbibitem

\bibitem[\protect\citeauthoryear{Zhang and Zhou}{2016}]{Zhang16}
\begin{barticle}[author]
\bauthor{\bsnm{Zhang},~\bfnm{{A. ~Y. }}\binits{A.}} \AND
  \bauthor{\bsnm{Zhou},~\bfnm{{H. ~H. }}\binits{H.}}
(\byear{2016}).
\btitle{Minimax rates of community detection in stochastic block models}.
\bjournal{Ann. Statist.}
\bvolume{44}
\bpages{2252--2280}.
\bdoi{10.1214/15-AOS1428}
\end{barticle}
\endbibitem

\end{thebibliography}

\end{document}